\documentclass[12pt,reqno]{article}
\usepackage{amssymb,amsmath,latexsym,color,graphicx}
\usepackage[colorlinks=true, pdfstartview=FitV, linkcolor=cyan, citecolor=blue, urlcolor=blue]{hyperref}
\usepackage{amsfonts,ulem}
\usepackage{amsthm}
\usepackage{amsmath}
\usepackage{stmaryrd}
\usepackage{yfonts}

\usepackage{graphicx}
\usepackage[margin=2.0cm]{geometry}
\usepackage{color}
\usepackage{cancel,tikz,framed,mathrsfs,arydshln}
\usepackage{ tikz, pgfplots, todonotes,color}
\usetikzlibrary{intersections, decorations.markings,arrows,calc, shapes}
\usetikzlibrary{decorations.markings,arrows, calc}
\usetikzlibrary{shapes.geometric,positioning}
\tikzstyle directed=[postaction={decorate,decoration={markings,
		mark=at position .65 with {\arrow{latex}}}}]
\def\remove#1 {{\color{red} \sout{#1}}}
\def \bb{\mathfrak b}
\def\m{\mathop}

\def\replace #1#2 { { \color{red} \tiny #1} {\color {blue} #2} } 
 \renewcommand{\theequation}{\thesection.\arabic{equation}}
\newtheorem{dref}{Definition}[section] 

\newtheorem{theo}[dref]{Theorem} 
\newtheorem{prop}[dref]{Proposition}
\newtheorem{remark}[dref]{Remark}

\def\Ccal{{\mathcal C}}
\newcommand{\len}{\ell}
\newcommand{\Pcal}{\mathcal{T}}

\newcommand{\Gcal}{\Gamma}
\newcommand{\Id}{\mathbb{I}_2}

\def\et{\widetilde{e}}
\def\eh{\widehat{e}}

\def\nt{\widetilde{n}}
\def\nh{\widehat{n}}
\def\gh{\h{g}}
\def\gt{\t{g}}
\def\Ctt{\widetilde{C}}
\def\Chh{\widehat{C}}
\def\Cp{C'}
\def\h{\widehat}
\def\t{\widetilde}
\def\g{\gamma}

\def\Mcal{\mathcal M}
\def\V{\mathcal V}
\def\g{\gamma}
\def\l{\lambda}
\def\be{\begin{equation}}
\def\ee{\end{equation}}
\def\la{\label}
\def\tr{{\rm tr}}
\def\d{{\mathrm{d}}}

\def\a{\alpha}
\def\b{\beta}
\def\R{{\mathbb R}}
\def\C{{\mathbb C}}
\def\Hcal{{\mathcal H}}

\def\at{\widetilde{\alpha}}
\def\bt{\widetilde{\beta}}
\def\ah{\widehat{\alpha}}
\def\bh{\widehat{\beta}}
\def\bea#1\eea{\begin{align}#1\end{align}}

\def\Ht{\widetilde{\Hcal}}
\def\Hh{\widehat{\Hcal}}
\def\Ch{\widehat{\Ccal}}
\def\Ct{\widetilde{\Ccal}}

\begin{document}

\vspace{0.2cm}
\begin{center}
\begin{huge}
{Log-canonical coordinates on $SL(2, \C)$ character varieties of compact Riemann surfaces}

\end{huge}
\bigskip
M. Bertola$^{\dagger}$\footnote{marco.bertola@concordia.ca},  
D. Korotkin$^{\dagger}$ \footnote{dmitry.korotkin@concordia.ca},
J.Pillet$^{\dagger}$ \footnote{jordi.pillet@concordia.ca}
\\
\bigskip
\begin{small}
$^{\dagger}$ {\it   Department of Mathematics and
Statistics, Concordia University\\ 1455 de Maisonneuve W., Montr\'eal, Qu\'ebec,
Canada H3G 1M8}
\end{small}
\vspace{0.5cm}
\end{center}

\vskip1.0cm

{\bf Abstract.}
We construct new sets of log-canonical coordinates on the $SL(2, \mathbb{C})$ character variety of compact Riemann surfaces. These  are labelled by  families of $1\leq m\leq 3g-3$ non-intersecting simple loops on
the Riemann surface and are
obtained by combining  complexified shear-type with length/twist-type coordinates.
In the case $m=3g-3$ the loops define a trinion decomposition of the Riemann surface, and  our coordinates are closely related to the (complexified)  Fenchel-Nielsen ones.

\vskip1.0cm

\tableofcontents 

\section{Introduction}
The $SL(2,\C)$ character variety, $\V_g$, of compact Riemann surfaces of genus $g$ is equipped with a  natural complex symplectic form, inverse to the Goldman Poisson bracket \cite{goldman1984symplectic}.  Local Darboux coordinates for this symplectic structure can be  obtained by analytic continuation of the Fenchel-Nielsen coordinates on the Teichm\"uller space $\Pcal_g$ \cite{wolpert1982, wolpert1985weil}  to the complex domain \cite{Platis}.

On the character variety, $\V_{g,n}$,  of a Riemann surface with  $n\geq 1$ boundary components the (complexified) Thurston's shear coordinates \cite{thurston} are log-canonical for the Goldman form on the symplectic leaves 
labelled by the values of complexified lengths of the boundary components
(we remind that a set of  coordinates is called log-canonical if the coefficients of the symplectic 2-form  are constant when written in terms of their logarithms).

In the case without boundary components,  $\V_g = \V_{g,0}$, however, there is so far no definition of complexified shear coordinates although important steps in this direction have been made in \cite{bonahon2018goldman}.  In the case of $\V_{2}$, a set of log-canonical coordinates alternative to Fenchel-Nielsen were recently constructed in \cite{chekhov2023symplectic}.

The main goal of this paper is  to address the issue in a general way and define log-canonical coordinates on $\V_g$ ($g\geq 2$)  by combining  the (complexified) shear coordinates and the twist-length coordinates. 
More specifically, to any system of simple non-intersecting closed contours $\{\gamma_1,\dots,\gamma_m\}$, $1\leq m \leq 3g-3$,  we associate a new system of (local) log-canonical coordinates on $\V_g$. In the  maximal case  $m=3g-3$, the set of contours defines a ``trinion decomposition'' of the Riemann surface into three-holed spheres (the {\it trinions} or {\it pairs of pants}).  In this case, our coordinates are closely related (but not exactly coinciding) with the complexified Fenchel-Nielsen length-twist coordinates.

Let us give some detail of our construction in the    case of one separating    loop $\gamma=\gamma_1$. Then    $\Ccal\setminus \gamma$ is the union of  two  Riemann surfaces, $\Ch$ and $\Ct$ of genera $\gt$ and $\gh$, respectively, with one boundary component each. 
Consider a homomorphism $ \rho:\pi_1(\Ccal, p_0) \to SL(2,\C)  $.
  Let $-\l_\g$ and $-\l_\g^{-1}$ be the eigenvalues of the monodromy matrix $\rho(\g)=M_\g$, and denote $\len_\g=\ln\l_\g$, with some choice of branch for the complex logarithm. Let us choose a point $\t v\in \Ct $ in a neighbourhood of $\gamma$ on the $\Ct$ side, and similarly another point $\h v\in \Ch$ that is also close to $\gamma$ but lying on the $\Ch$ side.  Consider the surface without boundary obtained by pasting back a disk on the boundary of $\Ct$ (we will refer to this surface in the sequel as the {\it capped} surface), and then consider a   triangulation  $\t{\Sigma}$ of the resulting surface  with only one vertex at $\t{v}$ such that the pasted disk is contained entirely inside of one of the triangles.
Let $\h{\Sigma}$ be a similarly constructed triangulation of $ \h \Ccal $ with a single vertex at $\h v$.

The triangulations $\t\Sigma$ and $\h \Sigma$ have  $\t n=6\t{g}-3$ and $\h n=6\h{g}-3$ edges,  respectively, and  the valences of the vertices $\t{v}$ and $\h{v}$ are equal to $2\t n$ and $2\h n$, since every edge of $\t\Sigma$ (resp. $\h\Sigma$) comes to $\t v$ (resp. $\h v$) twice.
To each edge, $\et_i$,  of the  graph $\t\Sigma$ we assign a  coordinate $\t \zeta_{e_i} \in \C$ and  $\t z_{e_i}= {\rm e}^{\t \zeta_{e_i}}$, which end up being interpreted as complex coordinates of shear type  parametrizing the character variety $\V_{\gt,1}$ of $\Ct$ (by abuse of language we continue calling $\t \zeta_{e_i}$  shear or shear type coordinates  although they really coincide with logarithms of Thurston's shear coordinates only when the length of the boundary curve is zero). As  shown in \cite{bertola2023extended}, see also Sec.\ref{CtCh}, the complex ``length''  $\ell_\gamma$ of the separating loop $\gamma$ is related to these coordinates by 
\be
2\sum_{i=1}^{\nt}  \t \zeta_{e_i}=\len_\gamma \mod 2\pi i
\la{zl}
\;.
\ee

Similarly,  coordinates $\h \zeta_{e_j} $'s and $\h z_{e_j}$'s are assigned to the edges $\eh_j$ of $\h\Sigma$ and they satisfy the relation
\be
2\sum_{j=1}^{\nh} \h \zeta_{e_j}=\len_\gamma \mod 2\pi i \;,
\la{yl}
\ee 
where it is important to point out that the parameter $\ell_\gamma$ in (\ref{zl}), (\ref{yl}) is the same.
There is an extra coordinate, $\beta_\gamma$, in the construction,  playing the role of a ``twist-parameter'' of the gluing around the contour $\gamma$.

The total number of independent variables, keeping into account the two linear constraints \eqref{zl}, \eqref{yl} 
is  $\nt+\nh=6g-6= {\rm dim} \V_g$.

As we show in  Th.\ref{maintheo} the symplectic form $\Omega$  on $\V_g^{(\gamma)}\subset \V_g$\footnote{The superscript denotes the subset of representations where $M_\gamma$ is diagonalizable.}  can be written as follows:
\be
  \Omega=\Omega_{0}+ {\Omega_{\t v}}+{\Omega_{\h v} }\;,
  \la{WPint1}
  \ee
where 
$$
    \Omega_0=\mathrm{d}\beta_\gamma \wedge \mathrm{d}\len_\gamma  \;,
$$
$$
{\Omega_{\t v}}= \sum_{\substack{i,j=1 \\ i<j} }^{12 \t{g}-6} 
 \mathrm{d}\zeta_{\t h_i} \wedge \mathrm{d} \zeta_{\t h_j}\;,\hskip0.7cm
{\Omega_{\h v} }=\sum_{\substack{k,l=1 \\ k<l} }^{12 \h g-6}  \mathrm{d}\zeta_{\h h_k} \wedge \mathrm{d}\zeta_{\h h_l}\;.
$$
The form is restricted to the (log)-linear constraints (\ref{zl}), (\ref{yl}). Here $\t h_j$ denote a local counterclockwise enumeration of the $12\t g-6$ edges  incident at $\t v$, since each of the $6\t g-3$ edges  $e_j$ is incident to $\t v$ at both ends. Similarly for $\h h_j$.  

This construction extends to an arbitrary system of simple non-intersecting contours ${\gamma_1,\dots,\gamma_m}$, $1 \leq m \leq 3g-3$, which split $\Ccal$ into $n$ subsurfaces $\Ccal^{(i)}$ with $k^{(i)}$ boundary components (notice that $\sum_{i=1}^n k^{(i)} = 2m$). In this case we have $m$ pairs of log-canonical coordinates $\{\beta_{\g_j},\len_{\g_j}\}_{j=1}^m$, one  for each loop $\gamma_j$.  
Denote by $\Sigma^{(i)}$ a triangulation of $\Ccal^{(i)}$ with $k^{(i)}$ vertices placed in neighbourhoods of the boundary components of $\Ccal^{(i)}$.
The number of edges of  $\Sigma^{(i)}$ is equal to  $6g^{(i)}-6+3k^{(i)}$; to each edge   we assign a  coordinate $\zeta_e$. In terms of these coordinates one can write the Goldman symplectic form on $ \V_g^{(\boldsymbol \gamma)}$ as follows (Th.\ref{MC}):
\be
\Omega=\sum_{i=1}^n \Omega^{(i)} + \sum_{j=1}^m \d\beta_{\g_j} \wedge \d\len_{\g_j} \;,
\la{O121}\ee
where
\begin{equation}
\Omega^{(i)}=\sum_{v\in V(\Sigma^{(i)})}\sum_{\substack{e, \t{e} \perp v \\ e <\t{e}} } \d\zeta_e \wedge \d \zeta_{\t{e}}\;;
\la{Oi1}
\end{equation}
 in the latter expression the form must be restricted to the subspace of the constraints that replace \eqref{zl}, \eqref{yl} in this case (a pair of such relations for each loop $\gamma_j$). 
 
When $ m=3g-3$, the surface $ \Ccal$ decomposes into $2g-2$ trinions ${\mathcal T}^{(i)}$, $i=1,\dots,2g-2$. A  trinion  contains 3 edge variables and 3  boundary lengths $ \len^{(i)}_1$, 
$ \len^{(i)}_2$, $ \len^{(i)}_3$,
thus all edge coordinates $\zeta_e$ can be expressed in terms of the boundary lengths  of ${\mathcal T}^{(i)}$
$$
2\zeta_1=-\len_1+\len_2+\len_3\;,\;\;\;
2\zeta_2=\len_1-\len_2+\len_3\;,\;\;\;
2\zeta_3=\len_1+\len_2-\len_3\;,\;\;\;
$$
which are valid for each trinion, thus the upper index is omitted.
Using the ribbon trinion graph $\mathfrak T$ (Def. \ref{deftringraph}), whose vertices correspond to trinions with a choice of cyclic order of the three outlets, and edges to the glueings of their boundary components, one may rewrite $\Omega$ as follows (see Theorem \ref{TheoTD}). 
  \begin{equation}\label{OMT}
    \begin{split}
       \Omega=   \sum_{\mathcal T^{(j)}\in V(\frak T) } (\d \len^{(j)}_1\wedge\d \len^{(j)} _2 + \d \len^{(j)} _2\wedge\d  \len^{(j)}_3+ \d \len^{(j)}_3\wedge \d \len^{(j)}_1)
        + \sum_{\frak e \in E(\frak T)} \mathrm{d} \beta_{\frak e}  \wedge \mathrm{d} \len_{\frak e} \;,
    \end{split}   
    \end{equation}
where $\beta_{\mathfrak e}$ are toric variables $\beta_{\mathfrak e}$ corresponding to each edge
  ${\mathfrak e}$ of the trinion graph;
$\len^{(j)}_a$, $a=1,2,3$ are the "complex lengths" of the boundary components of trinion $\mathcal T^{(j)}$, such that their ordering agrees with the ordering of the edges of the triangulation of the trinion as shown in Fig \ref{pants}. For each edge $\frak e=[j,k]$ that connects the boundary $a$ of  $\mathcal T^{(j)}$ and the boundary $b$ of $\mathcal T^{(k)}$  ($a, b\in \{1,2,3\}$) we impose the constraint $\d \ell_{a}^{(j)} =  \d \ell_{\mathfrak e}$ and $\d \ell_b^{(k)} = \d \ell_{\mathfrak e}$.

\begin{figure}[h!]
\centering
    \includegraphics[scale=0.8]{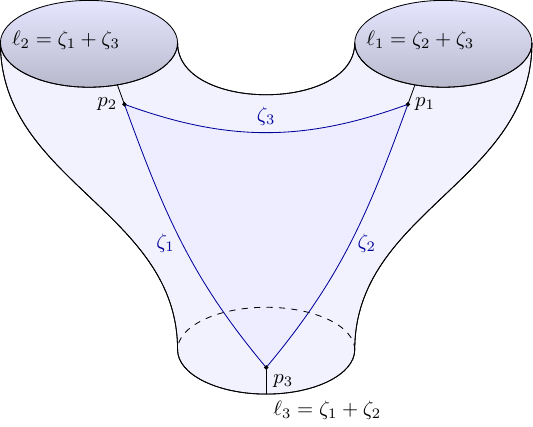}
    \caption{Trinion with triangulation and relation between edge variables and complex lengths of the boundaries }
    \label{pants}
\end{figure}

Explicit computations of the Goldman symplectic form $ \Omega$ for $V_g$ with $g=2$ are given in  appendix \ref{examplesapp}.

{\bf Acknowledgements.}  We thank Misha Shapiro for enlightening discussions which  were the starting point of  this work.
The work of M.B. was supported in part by the Natural Sciences and Engineering Research Council of Canada (NSERC) grant RGPIN-2016-06660. M. B. completed the work during his tenure as Royal Society Wolfson Visiting Fellow (RSWVF/R2/242024) at the School of Mathematics at the University of  Bristol. 
The work of D.K. was supported in part by the NSERC grant
RGPIN-2020-06816.

\section{Background}
\subsection{Character variety of a compact Riemann surface}
Let  $\Ccal$ be an oriented compact Riemann surface of genus $g$. Standard generators
$\{\alpha_k,\beta_k\}_{k=1}^g$ of the fundamental group
  $\pi_1(\Ccal)$ can be chosen so that they satisfy the fundamental relation in $\pi_1$:
\be
   \prod_{i=1}^g \left[ \alpha_i, \beta_i \right]= id \;,
   \la{relfun}
   \ee
where $$\left[ \alpha_i, \beta_i \right] = \alpha_i \beta_i^{-1} \alpha_i^{-1} \beta_i \;.$$
A point of the $SL(2, \C)$ character variety  $\V_g$ is  represented by the equivalence class under conjugation of the  set of monodromy matrices $\{M_{\alpha_1}, M_{\beta_1}..., M_{\alpha_g}, M_{\beta_g} \} \subset SL(2, \C)$ satisfying the represented version of \eqref{relfun}:
    \begin{equation}\label{repcl}
        \prod_{i=1}^g \left[ M_{\alpha_i}, M_{\beta_i} \right]=\mathbb{I}_2\;.
    \end{equation}
The ring of functions on  $\V_g$ is generated by the traces of all possible words in the alphabet of the generators $M_{\alpha_1}, \dots, M_{\beta_g}$.

The Goldman  Poisson bracket on $\V_{g}$ (the complexification of the Weil-Petersson bracket on $\Mcal_g$) can be written as follows in terms of trace functions
\begin{equation}
\label{GoldBrack}
\{ \tr M_{\gamma}, \tr M_{\Tilde{\gamma}} \} = \sum_{p \in \gamma \cap \Tilde{\gamma} } \nu(p) \left(\mathrm{tr}M_{\gamma \circ_p \Tilde{\gamma} }- \frac{1}{2} \tr M_{\gamma} \,\tr M_{\Tilde{\gamma}} \right) \;,
\end{equation}
where $ \gamma, \Tilde{\gamma} \in \pi_1(\Ccal)  $ and $\nu(p) $ is the contribution of the point $p$ to the intersection index of $\gamma$ and $\Tilde{\gamma}$.

The Goldman bracket on the character variety of {\it closed} surfaces, $\V_g$, is non-degenerate while for surfaces with boundary it has Casimir subalgebra  generated  by the traces of powers of loops retractable to the boundary components. The symplectic form inverting  the  Goldman's bracket on an arbitrary  $SL(N, \mathbb{C})$ character variety of a compact Riemann surface was found in \cite{alekseev1995symplectic}, but  corresponding log-canonical coordinates for them are not known.

Below, we show how to construct systems of Darboux coordinates on $\V_g$ which can be naturally extended to higher $SL(N, \mathbb{C})$ character varieties using so-called Fock-Goncharov coordinates.

\subsection{Extended character varieties for Riemann surfaces with boundary}

\label{SecExtSymp}
Let now $\Ccal$ be a closed and oriented Riemann surface of genus $g$ with $n \geq 1$  marked points $t_1,\dots,t_n$. Then the standard 
generators $\{\gamma_j\}_{j=1}^n$,
$\{\alpha_k,\beta_k\}_{k=1}^g$ of the fundamental group
  $\pi_1(\Ccal \setminus \{t_1,\dots, t_n\})$ can be chosen such that:
$$
 \quad \gamma_1...\gamma_n \prod_{i=1}^g \left[ \alpha_i, \beta_i \right]= id \;.
 $$

Consider now a representation  of $\pi_1(\Ccal \setminus \{t_1,\dots, t_n\})$ in $SL(2, \C) $.
Then monodromy matrices corresponding to the standard generators of the fundamental group  satisfy the same relation:
\be
M_{\gamma_1} ... M_{\gamma_n} \prod_{i=1}^g  \left[ M_{\alpha_i}, M_{\beta_i} \right] =\mathbb{I}_2 \;. 
\la{relMn}
\ee
For any $\gamma\in \pi_1(\Ccal \setminus \{t_1,\dots, t_n\})$ we denote the eigenvalues of  $M_{\gamma}$  by $ -\lambda_{\g}$ and $ -1/\lambda_{\g}$ and set  $\ell_{\gamma_j} := \ln \lambda_{\gamma_j}$ for some choice of determination.

Define the open subvariety $\m{\V}^0_{g,n} $ of the character variety $\V_{g,n}$ as the space of monodromy representations modulo simultaneous conjugation by $SL(2,\mathbb{C})$ such that each monodromy $M_{\g_k}$ has distinct eigenvalues. The ring of functions  on $\V^0_{g,n}$ is generated by the trace functions, i.e. $\tr\, M_\gamma$, $\gamma\in \pi_1(\Ccal \setminus \{t_1,\dots, t_n\})$.
The Goldman bracket    (\ref{GoldBrack}) on $\V_{g,n}^0$  is degenerate, with the Casimir elements being the eigenvalues $\lambda_{\gamma_j}$  of $M_{\gamma_j}$.

The variety $\V^0_{g,n}$ admits natural extension, ${\Hcal}_{g,n}$, which possesses a  {\it nondegenerate} Poisson structure. 
The idea of such an extension is to introduce canonical partners to the Casimirs $\len_{\g_j}= \ln \l_{\g_j},  \ j=1,\dots, n$ (we do not fix the order of the eigenvalues and eigenvectors of $M_{\gamma_k}$). 

Define the extended space $ \mathcal{H}_{g,n}$ as follows
\begin{equation}\label{ext}
{\Hcal}_{g,n}=\left \{ \{ M_{\alpha_j}, M_{\beta_j} \}_{j=1}^g, \{ C_k, \Lambda_k \}_{k=1}^{n} : \hskip0.4cm C_1 \Lambda_1 C_1^{-1}\dots C_N \Lambda_N C_N^{-1} \prod_{i=1}^g \left[ M_{\alpha_i}, M_{\beta_i} \right]=\mathbb{I}_2 \right \}\Big / \sim \;,
\end{equation} 
where $\Lambda_k $ are diagonal matrices in $SL(2,\C)$.  
The equivalence relations is  $M_{\gamma_k} \to G M_{\gamma_k} G^{-1}$, $C_k \to G C_k $ for a fixed $G \in SL(2, \C) $ (while $\Lambda_k$ remain invariant).

The space ${\Hcal}_{g,n}$, introduced in \cite{jeffrey1994extended}, is a toric fibration over $\V_{g,n}^{0}$.
The toric action consists in multiplication of the eigenvector matrices $C_k$ from  the right by an invertible diagonal $SL(2)$ matrix ${\rm diag} (b_{\gamma_k}, b^{-1}_{\g_k})$, where the  $b_{\gamma_k}$'s are called {\it toric variables}.

The  toric variables define the change of  normalization of the eigenvectors of $M_{\g_k}$ encoded in the matrices $C_k$: they can be used as coordinates as long as some reference normalization is used. This is done, in specific cases, in \cite{bertola2023extended}.

 The extended character variety ${\Hcal}_{g,n}$  possesses a natural symplectic form \cite{boalch2018wild,bertola2023extended}.
 In \cite{bertola2023extended} this symplectic form was represented in a log-canonical form for an arbitrary $SL(N, \mathbb{C})$ group by combining Fock-Goncharov coordinates with the toric variables. The toric coordinates $ \beta_{\g_k}= \ln b_{\g_k}$ are then the canonical partners of the $\len_{\g_k}$'s. 
 
\subsection{Graphs on surfaces and canonical symplectic form}

We summarize here the main ideas of \cite{bertola2023extended}.
Let $\Gcal$ be a finite 
 graph embedded into a Riemann surface $\Ccal$, with $V(\Gcal)$ and $E(\Gcal)$ denoting the sets of its (finitely many) vertices and oriented edges respectively. 
We will use the following notation: $e=[v_1, v_2]$ denotes the oriented edge connecting $v_1$ to $v_2$ and $-e=[v_2, v_1]$ denotes the oriented edge connecting $v_2$ to $v_1$. 
If an  edge $e$ is incident to a vertex $v$ we write  $e \perp v$.  We define the valence, $ n_v$, of a given vertex $ v$ to be the total number of edges incident to $v$. Note that an edge may be incident to the same vertex at both ends, in which case it contributes $2$ to the valence.

\begin{dref}\rm
\label{can graph}
    A pair $(\Gcal, J)$ consisting of an embedded graph $\Gamma$ on $\Ccal$, up to isotopy, and a map $J: E(\Gcal) \rightarrow SL(2, \C)$  is called {\it admissible}\footnote{In \cite{bertola2023extended} the term was "canonical" which we prefer to change here not to overload the term.}  if it satisfies the following conditions: 

\begin{enumerate}
 \item
 For any  $e\in E(\Gcal)$ we have  $J(-e)=J(e)^{-1}$, where $-e$ is the edge in the reversed orientation. We call  $ J(e)$  the {\it jump matrix} associated to the edge $ e$.
    \item
    For each vertex $v \in V(\Gcal)$ of valence $n_v \geq 2$ with incident {\it outgoing} edges $\{ e_1, e_2,..., e_{n_v}\}$ (enumerated counterclockwise starting from an arbitrary chosen edge) the {\it local monodromy} around $v$ is trivial, i.e. 
     \be
     J(e_1) J(e_2)...J(e_{n_v})=\mathbb{I}_2\;.
     \la{JJJ}\ee
    \end{enumerate}
\end{dref}

It is to be noted that the same edge $e\in E(\Gamma)$ may be incident to a vertex at both ends: of course,  in this case, the same matrix appears in \eqref{JJJ} twice,  once with power $1$ and once with power $-1$ somewhere in the cyclic product. 

To each admissible pair $(\Gcal, J) $ we associate the \textit{canonical  two-form} $ \Omega(\Gamma) $ as in \cite{bertola2023extended}:   
\begin{equation} \label{OmegaG}
        \Omega ({\Gcal})= \sum_{v \in V(\Gcal)} \Omega_v, \qquad \ \ \ 
        \Omega_v:= \frac{1}{2} \,\tr \sum_{l=1}^{n_v-1} \left(J_{[1...l]}^{(v)}\right)^{-1} \d J_{[1...l]}^{(v)} \wedge \left(J_l^{(v)}\right)^{-1} \d J_l^{(v)} \;,
\end{equation}  
where $J_l^{(v)}=J(e_l)$, $ e_l \perp v$, are the matrices defined by the map in Definition \ref{can graph} corresponding to the edges $e_1, e_2,..., e_{n_v}$ incident to $v$, oriented away from $v$  and 
$$
J_{[1...l]}^{(v)} = J_1^{(v)} J_2^{(v)}...J_l^{(v)} \;.
$$
The form $\Omega$ is independent of the choice of cyclic order at the vertices (as long as the no-monodromy condition \eqref{JJJ} holds at all vertices): moreover it is a {\it closed} two-form, 
\be
\d \Omega =0
\la{dOm}\ee
thanks to the following identity   holding for any set of  matrices $J_1, \dots, J_k$ \cite{alekseev1995symplectic}:
$$
\d\left( \tr \sum_{l=1}^{k-1} \left (J_{[1...l]} \right)^{-1} \d J_{[1...l]} \wedge \left(J_l\right)^{-1} \d J_l \right)= \tr \Big(
\mu^{-1} \d \mu \wedge\mu^{-1} \d \mu \wedge\mu^{-1} \d \mu\Big)\;,
$$
where
$$
\mu := J_1\cdots J_k\;.
$$
Thus, due to \eqref{JJJ}, the term corresponding to each vertex in the sum (\ref{OmegaG}) is a closed two-form, implying (\ref{dOm}).

The form  $ \Omega({\Gcal}) $ is invariant under the  following {\it admissible moves} \cite{bertola2023extended}:
\begin{enumerate}
    \item {\it merging} of two vertices $ v_1$ and $ v_2$ of valence $ n_{v_1}+1$ and $ n_{v_2}+1$ respectively, connected by a common edge $ e_0$, to  a single vertex $ v$ of valence $ n_{v_1}+n_{v_2}$.
    \item {\it zipping} of two edges $ e_1$ and $ e_2$ connecting the same two vertices $ v_1$ and $ v_2$, to a single edge $ e $ with jump matrix $ J(e) =  J(e_1)J(e_2)$. \\
\end{enumerate}
\paragraph{From admissible pairs to representations of $\pi_1$.}
The data of an admissible pair $(\Gamma, J)$ defines also a representation of the fundamental group of $\Ccal\setminus V_1(\Gamma)$ (where $V_1(\Gamma)$ is the collection of univalent vertices of $\Gamma$) as follows; 
fix a basepoint $z_0\in \Ccal \setminus \Gamma$. For each homotopy class   $[\gamma]\in \pi_1\big(\Ccal\setminus V_1(\Gamma), z_0\big)$ choose a representative that intersects $\Gamma$ at the edges and transversally. 
Then define 
\be
\label{defrho}
\rho(\gamma) = \prod_{p\in \gamma \cap \Gamma} J(e_p)^{\sharp(e_p,\gamma)}\;,
\ee
where $e_p$ is the edge intersected by $\gamma$ at $p\in\Gamma\cap \gamma$ and $\sharp(e_p,\gamma)\in \{1,-1\}$ is the orientation of the intersection ($+1$ if the tangent to $e_p$ and that of $\gamma$ at $p$ form a positively oriented  frame). 
Thanks to \eqref{JJJ} this defines a representation of the fundamental group because it does not matter if we go around a multi-valent vertex on one side or the other, and, consequently, its adjoint equivalence class defines a point of the character variety.  
\subsection{Canonical dissection graph and symplectic form}

Let $ \Ccal$ be a compact  Riemann surface of genus $g$. Choose a base point $v_0$ and a set of standard generators $\{\a_j,\b_j\}_{j=1}^g$ 
of $\pi_1(\Ccal,v_0)$ satisfying (\ref{relfun}). Denote by $M_{\alpha_i}$, $M_{\beta_i}$ the corresponding monodromies
 which satisfy (\ref{relMn}) for $n=0$:
\be
\prod_{j=1}^g [M_{\a_j} , M_{\b_j}]=\mathbb{I}_2 \;.
\la{MMM1}
\ee

Let us choose some oriented closed contours on $\Ccal$ starting at $v_0$ representing the elements $\a_j$, $\b_j$; with a slight abuse of notations we denote these loops by the 
same letters. Consider the one-vertex  graph $\Gcal_0$ (which we call the \textit{canonical dissection graph}, see Figure \ref{fig:SG} (a)) with vertex at $v_0$ and $2g$ oriented edges 
$\{{\a_j},{\b_j}\}_{j=1}^g$; the valence of $v_0$ equals $4g$. Define the jump matrices on the edges of $\Gcal_0$ as follows:
\be
J({\alpha_j})=M_{\beta_j} \;, \hskip0.7cm J({\beta_j})=M_{\alpha_j} \;.
\la{JJMM}
\ee
Notice that the relation (\ref{JJJ}) in the general definition of jump matrices on a given graph is written under assumption  that all the edges $e_j$ are incoming to 
the vertex. In our case, since every edge meets the vertex twice, while defining the edge matrices $J(e_j)$ in the  form (\ref{OmegaG}) we need to reverse the orientation 
of $2g$ out of the  $4g$ half-edges incident to $v_0$ as shown  in Fig \ref{fig:SG} (b).

\begin{figure}[ht!]
\centering
\includegraphics{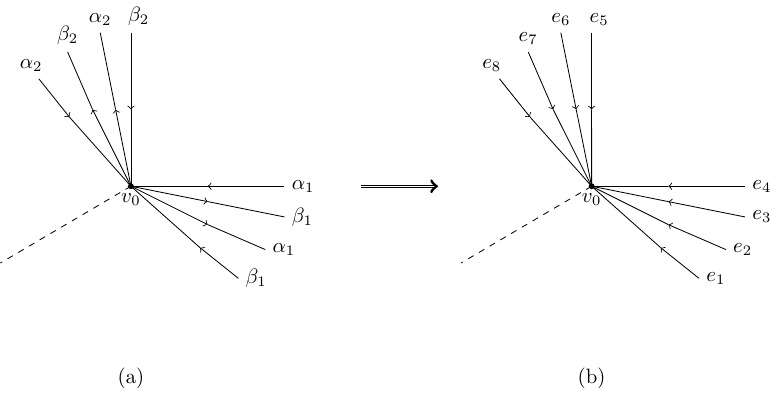}
\caption{Orientation of the edges of the canonical dissection graph $\Gcal_0$ for $ g=2$.}
\label{fig:SG}
\end{figure}

Then we get: 
$$
J(e_1)= M_{\alpha_1}\;, \hskip0.5cm
J(e_2)=M_{\beta_1}^{-1}\;, \hskip0.5cm
J(e_3)= M_{\alpha_1}^{-1}\;, \hskip0.5cm
J(e_4)=M_{\beta_1}\;, \dots \;.
$$

Under these conventions we obtain
        $$ J(e_1)J(e_2)...J(e_{4g})= \prod_{i=1}^{g}  M_{\alpha_i}  M_{\beta_i}^{-1} M_{\alpha_i}^{-1}  M_{\beta_i} =\Id\;. $$

Therefore, the canonical dissection  graph $\Gcal_0$ with our choice of jump matrices satisfies the condition (\ref{JJJ}). The graph $\Gamma_0$ has   one vertex of valence $4g$  and $2g$ edges.

Applying the formula (\ref{OmegaG}) to the pair $(\Gcal_0, J)$ we get the symplectic form $ \Omega(\Gamma_0)$ which can be written as follows:

\begin{equation}\label{OmegaSt}
    \Omega(\Gamma_0)= \frac{1}{2} \sum_{j=1}^{2g-1} \mathrm{tr}\left(K_{[1...j]}^{-1} dK_{[1...j]} \wedge K_j^{-1} dK_j \right) \;,
\end{equation}
where $K_j$ is the $j$-th matrix in the product (\ref{MMM1}); $ K_{[1...j]} $ denotes the ordered product of matrices entering in the same product (\ref{MMM1}) up to $j$; for instance 
$$K_{[1]}=M_{\alpha_1}\;,\hskip0.5cm    K_{[12]}=M_{\alpha_1} M_{\beta_1}^{-1}\;,\hskip0.5cm
 K_{[123]}= M_{\alpha_1} M_{\beta_1}^{-1} M_{\alpha_1}^{-1} \;, \dots \;.$$

\begin{prop}
    The form $ \Omega(\Gamma_0)$ given by (\ref{OmegaSt}) coincides with the Goldman  symplectic form $\Omega$ on $\V_g$.
\end{prop}
The proof is contained in  \cite{bertola2023extended} and we recall here the main steps:
   the symplectic form $\Omega$ inverting the Goldman bracket on the $SL(N, \mathbb{C})$ character variety of $\pi_1(\Ccal)$ was computed by
    Alekseev and Malkin \cite{alekseev1995symplectic}. The expression of   \cite{alekseev1995symplectic} was cast in  the canonical form $\Omega$   \eqref{OmegaG} associated to a particular  graph $G_{AM}$  in \cite{bertola2023extended} (see Fig.6 in loc.cit.). The Alekseev-Malkin graph $G_{AM}$   can be transformed to the canonical dissection graph $\Gcal_0$ by a series of admissible moves preserving the symplectic form.  Therefore,  the form $\Omega(\Gamma_0)$ coincides with the Goldman symplectic form (for an arbitrary $SL(N, \mathbb{C})$ case). 

\section{Parametrization of  $\V_g$ by  gluing along one dissecting contour}
\label{Sec5}
Here, we are going to construct another admissible pair  (in the sense of Def.\ref{can graph})  of graph and jump  matrices  which can be reduced  to $\Gcal_0$ (canonical dissection) by a series of admissible moves. This provides an alternative representation for the form $\Omega(\Gamma_0)$ which will be used in the next section  to construct 
new systems of log-canonical  coordinates on $\V_g$. 
In Sec.\ref{Sec5.1} we consider in detail the case where the dissecting contour disconnects the Riemann surface of genus $g$ into two surfaces of genera $\t g, \h g$, $\t g + \h g = g$, with one boundary each. 
In Sec.\ref{nonsep} we consider the non-separating case where the contour dissects the surface into a surface of genus $g-1$ with two boundaries.

\subsection{Oriented graph on $\Ccal$ by amalgamation of  graphs on two Riemann surfaces of lower genera }
\label{Sec5.1}
Consider two topological Riemann surfaces, $\Ct$ and $\Ch$, of genera $ \t{g}$ and $\h{g}$, respectively (such that $g=\t{g}+\h{g}$) with one boundary component each. We denote by $ \t{\gamma}$ (resp. $\h{\g}$) a contour in the neighbourhood of the boundary and retractable to it.

Pick a point $ \t{v} $ on  $\Ct$ near $\t{\g}$ and a point $ \h{v} $ on  $\Ch$  lying near $\h{\g}$.
The fundamental group  $\pi_1( \Ct,\t{v}) $ is freely  generated by the elements $ \{ \t{\alpha}_i, \t{\beta}_i \}_{i=1}^{\t{g}}$
and we define the monodromy around the boundary of $\Ct$ by
\be
M_{\t{\gamma}}^{-1}=\prod_{i=1}^{\t{g}} \left[ M_{\t{\alpha}_i}, M_{\t{\beta}_i} \right]\;.
\la{Mtdef}
\ee
We assume that $M_{\t{\gamma}}$ is diagonalizable, $M_{\t{\gamma}}=\t{C} \t{\Lambda} \t{C}^{-1}$ where the matrix of eigenvectors
$\t{C}$ has a freedom of multiplication by an  $SL(2, \C)$ diagonal matrix from the right.

Similarly  $\pi_1( \Ch,\h{v}) $ is freely  generated by the  elements $\{\h{\alpha}_j, \h{\beta}_j\}_{i=1}^{\h{g}}$ 
and we define the monodromy around the boundary of $\Ch$ by
\be
M_{\h{\gamma}}^{-1}=\prod_{i=1}^{\h{g}} \left[ M_{\h{\alpha}_i}, M_{\h{\beta}_i} \right]\;,
\la{Mhdef}\ee
which is also assumed to be diagonalizable,
 $M_{\h{\gamma}}=\h{C} \h{\Lambda} \h{C}^{-1}$.

Define the following  two instances of the extended character variety \eqref{ext}:
\begin{equation}\label{pieces}
  \Ht = \Hcal_{\t g, 1} =  \left \{ \left \{ M_{\t{\alpha}_j}, M_{\t{\beta}_j}  \right \}_{j=1}^{\t{g}},\;\; \{ \t{C}, \t{\Lambda} \} : \hskip0.4cm \t{C} \t{\Lambda} \t{C}^{-1} \prod_{i=1}^{\t{g}} \left[ M_{\t{\alpha}_i}, M_{\t{\beta}_i} \right]=\mathbb{I}_2 \right \}/\sim \;,
\end{equation}
\be
    \Hh =   \Hcal_{\h g, 1} = \left \{ \left \{ M_{\h{\alpha}_j}, M_{\h{\beta}_j}  \right \}_{j=1}^{\h{g}}, \{ \h{C}, \h{\Lambda} \} :  \hskip0.4cm \h{C}\h{ \Lambda} \h{C}^{-1} \prod_{i=1}^{\h{g}} \left[ M_{\h{\alpha}_i}, M_{\h{\beta}_i} \right]=\mathbb{I}_2 \right \}/\sim \;,
\la{pi1}
\ee 
where  $\t{C}, \t{\Lambda}\in SL(2, \C)$ and $\t{\Lambda}=\mathrm{diag}(-\lambda_{\t \g},-\lambda_{\t \g}^{-1})$. The sign here is chosen for  later convenience.
The  monodromies  $\{M_{\t{\alpha}_j}, M_{\t{\beta}_j}  \}$ can be chosen arbitrarily; they 
determine  the matrix $\t{\Lambda}$ up to the ordering of eigenvalues. Similarly for the    hatted    version.

For  a matrix $M_\gamma$ with eigenvalues $-\lambda_\gamma^{\pm 1}$ we will call $\ell_\gamma =  \ln \lambda_\gamma$ the ``oriented complex length'' (called simply ``length'' below). 
\vskip0.3cm

The idea of the following proposition is to construct an element of $\V_g$ starting from  two sets of matrices representing a point of 
$\t{\Hcal}$ and one  of  $\h{\Hcal}$,   satisfying the additional constraint 
\be
 \h{\Lambda} = \t{\Lambda}\;.
 \la{LL}
 \ee
 
 We set $\Lambda = \h \Lambda = \t\Lambda$ in the sequel.
Observe that for a diagonal matrix $\Lambda$, the inverse matrix is in the same adjoint orbit, namely,
\be
\Lambda^{-1} =  \bb^{-1} \Lambda \bb \;,
\ee
where
\be
 \bb =
\begin{bmatrix}
0&1\\-1&0
\end{bmatrix} \;. 
\label{defb}
\ee

Let us glue (topologically) the surfaces $\Ct$ and $\Ch$ along the boundary to get a  topological surface $\Ccal$ of genus $g=\t{g}+\h{g}$. We denote by $ \g$  the gluing contour  on $ \Ccal$ which lies between  $ \t{\g}$ with $\h{\g}$, see Fig.\ref{FIG32}.

\begin{figure}[!ht]
\centering
\includegraphics[scale=1.3]{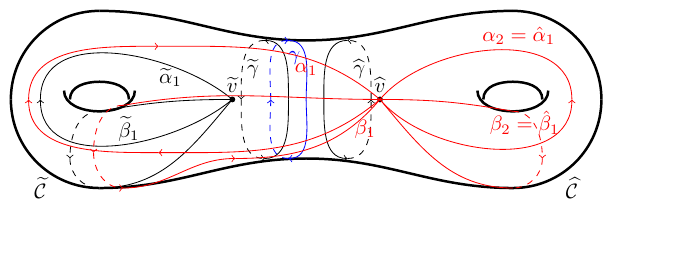}
\caption{Generators of the fundamental groups $\pi_1( \t{\Ccal},\t{v}) $ (black) and  $\pi_1( \h{\Ccal},\h{v}) $ and corresponding generators of $\pi_1(\Ccal,\h{v})$ (red) for $ g=2$.}
\label{FIG32}
\end{figure}

The generators  $\{\alpha_j,\beta_j\}_{j=1}^g$ of the fundamental group $\pi_1(\Ccal,\h{v})$ are naturally inherited from the generators
$\{\t{\alpha}_j,\h{\alpha}_k\}$ and  $\{\t{\beta}_j,\h{\beta}_k\}$ of the fundamental groups of $\Ct$ and $\Ch$ as shown in Fig.\ref{FIG32}.

 We now  construct a surjective    map 
\be
F:  \Ht\times \Hh \Big|_{\h{\Lambda}=\t{\Lambda} = \Lambda } \to \V_g^{(\gamma)}\;,
\la{mapF}
\ee
where $\Ht, \Hh$ are defined in \eqref{pieces}, \eqref{pi1} and  $\V_g^{(\gamma)}\subset \V_g$ consists of 
 representations for which the monodromy $M_\gamma$ is diagonalizable.
\begin{prop}\label{fact}

    Consider a set of matrices representing a point of  $\Ht $ and a set of matrices representing a point of  $ \Hh$ such that  $\h{\Lambda}=\t{\Lambda}$.
    
    The map \eqref{mapF} is defined by taking  the equivalence class,  under common adjoint action, of the  following matrices in  $SL(2, \C)$: 
    $$
           M_{\alpha_i} = \Chh  \bb^{-1} \Ctt^{-1} M_{\t{\alpha_i}} \Ctt  \bb \Chh^{-1} \;,\hskip0.6cm M_{\beta_i} = \Chh 
\bb^{-1} \Ctt^{-1}  M_{\t{\beta_i}} \Ctt \bb \Chh^{-1}, \quad i =1,...,\t{g} \;,
           $$
        \begin{equation} \label{monC}     
            M_{\alpha_{i+\t g}} = M_{\h{\alpha_i}} \;,\hskip0.6cm  M_{\beta_{i+ \t g}}=M_{\h{\beta_i}} \;, \quad i=1,...,\h{g} \;,
    \end{equation}
 which  satisfy the relation \eqref{repcl} (here $\bb$ is in \eqref{defb}).
    
    The map \eqref{mapF}, defined  by \eqref{monC} is surjective.
\end{prop}

\begin{proof} 
From the definitions (\ref{pieces}) of $ \t{\Hcal}$ and $ \h{\Hcal}$ one has
\begin{equation} \label{splitlambda}
       \t{\Lambda} = \t{C}^{-1} \left(\prod_{i=1}^{\t{g}} \left[ M_{\t{\alpha}_i}, M_{\t{\beta}_i} \right] \right)^{-1}\!\!\!\! \t{C} \;, \ \ \ 
        \h \Lambda = \h{C}^{-1} \left(\prod_{i=1}^{\h{g}} \left[ M_{\h{\alpha}_i}, M_{\h{\beta}_i} \right] \right)^{-1} \!\!\!\!\h{C}\;.
\end{equation}
    Imposing the constraint \eqref{LL}, $ \t{\Lambda} = \h{\Lambda} $, we get 
    \begin{equation}\label{com}
        \h{C} \bb^{-1} \t{C}^{-1} \left(\prod_{i=1}^{\t{g}} \left[ M_{\t{\alpha}_i}, M_{\t{\beta_i}} \right] \right) \t{C}  \bb \h{C}^{-1} \left( \prod_{i=1}^{\h{g}} \left[M_{\h{\alpha}_i}, M_{\h{\beta}_i} \right] \right)=\mathbb{I}_2 \; .
    \end{equation}
    Defining $ M_{\alpha_i} = \h{C} \bb^{-1} \t{C}^{-1} M_{\t{\alpha}_i} \t{C} \bb \h{C}^{-1} $ and $  M_{\beta_i} = \h{C} \bb^{-1} \t{C}^{-1}  M_{\t{\beta_i}} \t{C} \bb \h{C}^{-1} $ for $ i=1,..., \t{g}$, the equation (\ref{com}) can be rewritten as:
    $$ \prod_{i=1}^{\t{g}} \left[ M_{\alpha_i}, M_{\beta_i} \right] \prod_{i=1}^{\h{g}} \left[ M_{\h\alpha_i}, M_{\h\beta_i} \right] =\mathbb{I}_2\;.  $$
   Now, defining  $ M_{\alpha_{i+\t{g}}} = M_{\h{\alpha}_i} $ and $M_{\beta_{i+\t{g}}}=M_{\h{\beta_i}}$ for $i =1,..., \h{g} $, this last identity can be rewritten as follows:
    \bea
     \prod_{i=1}^{\t{g}} \left[ M_{\alpha_i}, M_{\beta_i} \right] \prod_{i=\t{g}+1}^{\t{g}+\h{g}} \left[ M_{\alpha_i}, M_{\beta_i} \right] =\mathbb{I}_2 \;, 
     \label{MMM}
     \eea
    which coincides with   (\ref{repcl}).

    To prove the surjectivity of this map onto $\V_g^{(\g)}$ one simply has to reverse the process: start from a given representation in $\V_g^{(\g)}$ and choose the same basis of cycles as in the proof above: 
then we have \eqref{MMM}. The matrix $M^{-1}_{\t{\gamma}}$ is, say, the first product of $\t g$ terms, and it is assumed to be diagonalizable, $M_{\t{\gamma}} = C \Lambda C^{-1}$:
$$
 \prod_{i=1}^{\t{g}} \left[ M_{\alpha_i}, M_{\beta_i} \right]  = C \Lambda C^{-1}\;.
 $$
 Similarly, 
 $$
  M^{-1}_{\h{\gamma}}= \prod_{i=1}^{\h{g}}  \left[ M_{\h\alpha_i}, M_{\h\beta_i} \right]  = C \Lambda^{-1}  C^{-1} = C \bb\Lambda  \bb^{-1} C^{-1}.
$$
Thus we read off the points in $\t\Hcal, \h\Hcal$ (with $\h C = C\bb$).
\end{proof}

Let us now use this gluing formalism to  construct the admissible graph $\Gcal_1$ on $\Ccal$  with jump matrices $J$ on its edges  such that the canonical form 
$\Omega(\Gcal_1)$ coincides with the Goldman   form $\Omega$ on $\V_g$.

The graph $\Gcal_1$ is constructed by amalgamation of three graphs (see  Fig. \ref{graphs}):
\begin{itemize}
    \item the graph $\t{\Gcal}_1$, which is a canonical dissection along a set  of generators for the fundamental group, with basepoint $\t v$, of the  capping of  $\t{\Ccal}$,
    \item similar graph   $\h{\Gcal}_1$  on $\h{\Ccal}$,
    \item a five-vertex ``gluing'' graph $\Gcal_{pl}$.
\end{itemize}

 The graph $\Gcal_{pl}$ has two one-valent vertices at $\t{v}$, $\h{v}$, and three four-valent vertices 
which we denote by $\t{q}$, $ q $ and $\h{q}$; these are the points where the segment connecting $\t{v}$ and $\h{v}$ crosses $\t{\g}$, $ \g$ and $\h{\g}$, respectively, see Fig. \ref{graphs}.

The jump matrix map  $J$  on the edges of  $\t{\Gcal}\cup \h{\Gcal}$ is given by 
$$ J(\at_j)  = M_{\bt_j}\,, \hskip0.7cm  J(\bt_j)=  M_{\at_j} \,, \hskip0.7cmJ(\ah_j)=M_{\bh_j}\,, \hskip0.7cm J(\bh_j)= M_{\ah_j}\;.$$
Finally, the jump matrices on the edges of $\Gcal_{pl}$  are given by
$$
    J([\t v, \t q])= J(\t{e}_1)=M_{\t{\g}}\;, \hskip0.7cm \quad J(\t{e}_2 )=J(\h{e}_2)=\Lambda\;, \hskip0.7cm J([\h v, \t q])=J(\h{e}_1)=M_{\h{\g}} \;, 
    $$
    $$
     J(\t{\g})=\Ctt\;, \hskip0.7cm J(\h{\g})=\Chh\;, \hskip0.7cm  J(\g)= \bb \;,
$$
as shown in Fig. \ref{GluingGraph} (recall that $M_{\t{\g}}$ and $ M_{\h{\g}}$ are defined by (\ref{Mtdef}) and (\ref{Mhdef})).

\begin{figure}[h!]
\centering
    \includegraphics[scale=1.1]{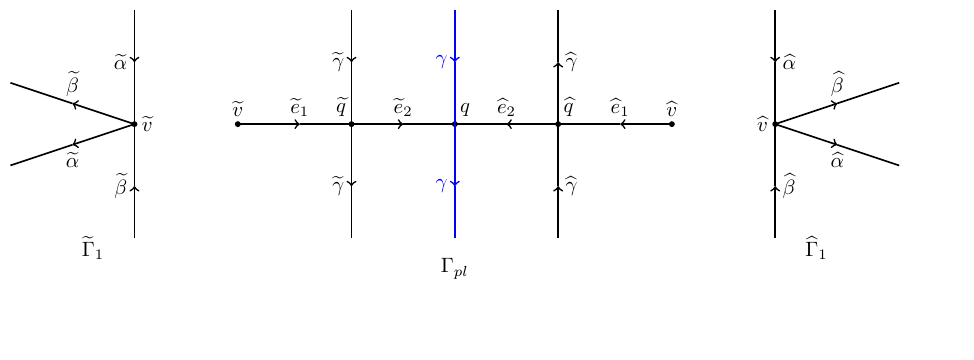}
    \caption{The graphs $ \t{\Gamma}_1$, $ \Gamma_{pl}$ and $ \h{\Gamma}_1$ for $ g=2$. }
    \label{graphs}
\end{figure}

\begin{figure}[!ht]
\centering
\includegraphics[scale=1]{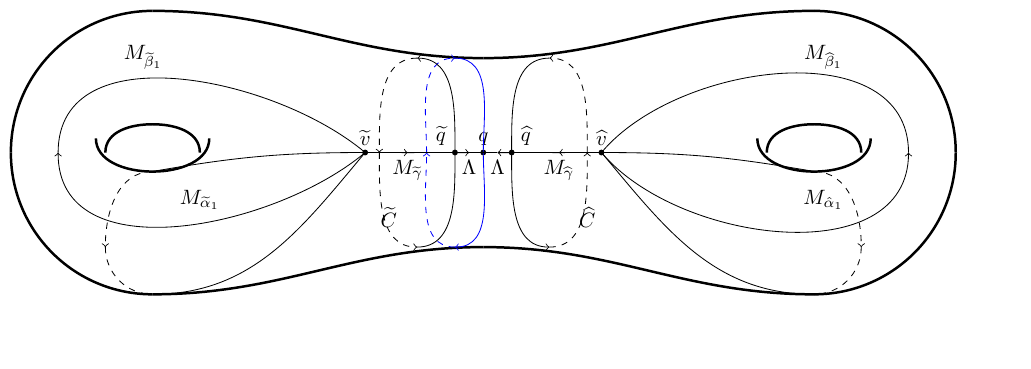}
\caption{The amalgamated graph $ \Gamma_1 $ drawn on $ \Ccal$ for $ g=2$.}
\label{GluingGraph}
\end{figure}

One can verify that the constraint \eqref{JJJ} is satisfied i.e. the total monodromy around each of the five vertices is trivial. Consider, for example, $ q$, $\t{q}$ and $\t{v}$ (the computation at $\h{v}$ and $\h{q}$ is identical).

\begin{itemize}
    \item For $ q$: 
  The four edges incident to $ q$ in counterclockwise order have jump matrices: $ J(e_1)=\bb^{-1}$, $ J(e_2)=\Lambda$, $ J(e_3)=\bb$ and $ J(e_4)=\Lambda$. Then
    $$ J(e_1) J(e_2) J(e_3) J(e_4)=\bb^{-1} \Lambda \bb \Lambda = \Lambda^{-1} \Lambda=\Id\;. $$
    \item For $\t{q}$:
the jump matrices are given by 
$$ 
J(e_1)=M_{\t{\g}}      \,, \hskip0.7cm         J(e_2)=\t{C}     \,, \hskip0.7cm         J(e_3)=\Lambda^{-1}\,, \hskip0.7cm     J(e_4)=\t{C}^{-1}\;,
     $$
  and
     $$ 
     J(e_1)J(e_2)J(e_3)J(e_4)=M_{\t{\g}} \t{C} \Lambda^{-1} \t{C}^{-1}=M_{\t{\g}}M_{\t{\g}}^{-1}=\Id\;.
     $$
     \item For $ \t{v}$:
   we have
    $$ 
    J(e_1)=M_{\t{\alpha}_1}\;, \hskip0.7cm J(e_2) = M_{\t{\beta}_1}^{-1} \;, \hskip0.7cm
    J(e_3)=M_{\t{\alpha}_1}^{-1}  \;, \hskip0.7cm   J(e_4)=M_{\t{\beta}_1}\;,
    $$ 
    $$
    J(e_5)=M_{\t{\alpha}_2}\,, \dots, \hskip0.7cm J(e_{4 \t{g}+1})=J( \tilde{e})=M_{\t{\g}}^{-1}\;.$$ 
    Therefore, 
    $$ \prod_{i=1}^{4\t{g}+1} J(e_i)= \prod_{i=1}^{\t{g}}\left[ M_{\t{\alpha}_i}, M_{\t{\beta}_i} \right] M_{\t{\g}}^{-1} = \Id\;. $$
\end{itemize}
The graph $ \Gamma_1$ can be transformed into the canonical dissection  graph $ \Gamma_0$ by a sequence of admissible moves in the sense of \cite{bertola2023extended}:
\begin{enumerate}
    \item Merge the vertices $ \t{v}$ with $\t{q}$ and $ \h{v}$ with $ \h{q}$.
    \item Merge the vertices $ q$ ,  $\t{q}$ and $ \h{q}$.
    \item Zip together the three edges $ \g $,  $ \t{\g}$ and $ \h{\g}$ whose jump matrices are respectively $ \bb$,  $ \t{C}^{-1} $ and $ \h{C}$ to get a single edge $ \g_0$ with jump matrix $ J(\g_0)=\h{C} \bb^{-1} \t{C}^{-1}$.
    \item Zip the edge $ \g_0$ with the edges $ \{ \t{\alpha}_i, \t{\beta}_i \}_{i=1}^{\t{g}}$ to get edges $ \{ \alpha_i, \beta_i \}_{i=1}^{\t{g}}$ with corresponding jump matrices: $$ J(\beta_i)= M_{\alpha_i} = \h{C} \bb^{-1} \t{C}^{-1} M_{\t \alpha_i} \t{C} \bb \h{C}^{-1} \;,\hskip0.8cm
    J(\alpha_i)= M_{\beta_i} = \h{C} \bb^{-1} \t{C}^{-1}  M_{\t \beta_i} \t{C} \bb \h{C}^{-1}$$
    for  $i=1,...,\t{g}$.    
    \end{enumerate}
As a result of this transformation we see that  the two canonical symplectic forms coincide:  $ \Omega(\Gamma_1)= \Omega(\Gamma_0)$. In other words, one can use the graph $ \Gamma_1$ to compute the Goldman  symplectic form $ \Omega $ on $\V_g $; further transformation of the graph $\Gamma_1$ will be discussed in the next section.
%{\color{blue}

\subsection{Oriented graph on $\Ccal$ from a  graph on a Riemann surface of lower genus with two boundaries}
\label{nonsep}

We now consider the case when $\gamma$ is a non-separating contour on $\Ccal$ i.e. dissection of  $\Ccal$ along $\gamma$ produces  a Riemann surface $\Ccal_{g-1}$ of genus $g-1$ with two boundaries. 
To construct the corresponding system of log-canonical coordinates on  $\V_g^{(\gamma)} $ 
we start from the space $\Hcal_{g-1,2}$ (see \eqref{ext}) and then prove the following analog of Prop.\ref{fact}:
\begin{prop}\label{fact2}
    Consider a set of matrices \eqref{ext} representing a point of  $\Hcal_{g-1,2}$ such that  $\Lambda_1=\Lambda_2=\Lambda$. Then we define the surjective map 
\be
F:  \Hcal_{g-1,2} \Big|_{\Lambda_1=\Lambda_2 = \Lambda } \to \V_g^{(\gamma)}
\la{mapF2}
\ee
 by constructing   the  set of  matrices $\{M_{\alpha_j}, M_{\beta_j}\}_{j=1}^g$ representing a point in  $\V_g^{(\gamma)}$ as follows: the matrices $\{M_{\a_j}, M_{\b_j}\}_{j=1}^{g-1}$  remain  the same as in $\Hcal_{g-1,2}$   and 
 \be
   M_{\alpha_g} = C_1  \Lambda C_1^{-1} = M_1;,\hskip0.6cm 
           M_{\beta_g} = C_1 \bb C_2^{-1}\;. 
 \la{MaMb}
\ee
    The map \eqref{mapF2}  is surjective.
\end{prop}
\begin{proof}
The proof is very similar to that of Prop.\ref{fact}. 
From the definition of $\Hcal_{g-1,2}$ we have 
\be
C_1 \Lambda_1 C_1^{-1} C_2 \Lambda_2 C_2^{-1} \prod_{j=1}^{g-1} \big[M_{\a_j}, M_{\b_j}\big] = \Id. 
\ee
Let us now  realize the pre-factor to the product as a commutator. Under  the constraint $\Lambda_1=\Lambda_2 = \Lambda$ one can write 
$$
C_1 \Lambda C_1^{-1} C_2 \Lambda C_2^{-1} = C_1 \Lambda C_1^{-1} C_2  \bb^{-1} C_1^{-1} C_1\Lambda^{-1}C_1^{-1} C_1 \bb C_2^{-1} = M_{\a_{g}} M_{\b_g}^{-1} M_{\a_g}^{-1} M_{\b_g}\;. 
$$
where $M_{\a_{g}} $  and $M_{\b_g}$ are given by (\ref{MaMb}).
The rest of the proof proceeds in complete analogy to that of Prop. \ref{fact}.
\end{proof}

Similarly to the separating case, we construct the graph $\Gamma_1$ on $\Ccal$ as follows. First, consider the canonical dissection of the curve $\Ccal_{g-1}$ along the set of generators of     
$\pi_1(\Ccal_{g-1},x_0)$ for some initial point $x_0$ with two edges connecting $x_0$ to two other points $\hat{v}$ and $\tilde{v}$. The plumbing graph with two one-valent vertices also denoted by
$\hat{v}$ and $\tilde{v}$ is then amalgamated with the above graph on  $\Ccal_{g-1}$  as shown in Fig.\ref{twoedges}. This gives the graph $\Gamma_1$ in the case of the non-separating cycle.
Similarly to the separating case, it can be deformed to  $\Gamma_0$ by a sequence of admissible moves, and, therefore, $\Omega(\Gamma_1)=\Omega(\Gamma_0)$.

\begin{figure}[h!]
\centering
    \includegraphics[scale=1.1]{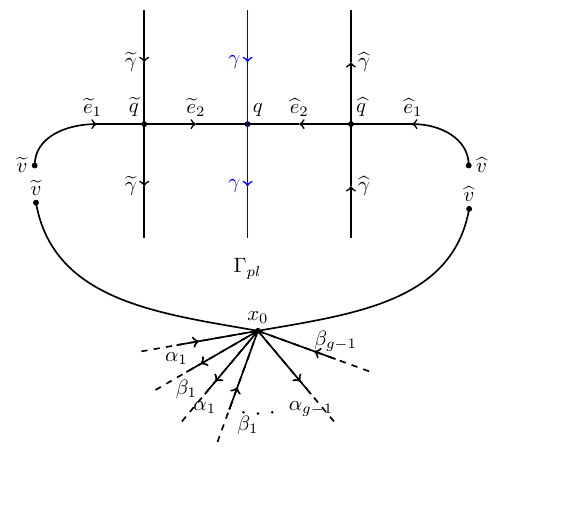}
    \caption{The graph $\Gamma_2$ in the case of one non-separating contour. }
    \label{twoedges}
\end{figure}

%}

\section{Triangulations of dissected surface and log-canonical  coordinates}
\label{CtCh}
\subsection{  Dissection along a separating contour }
\label{CtCh1}
For explicit computation of the Goldman  form in terms of  log-canonical coordinates we are going to construct one more admissible pair $(\Gamma_2,J)$  on $\Ccal$, equivalent to  the pairs $(\Gamma_1,J)$ and $(\Gamma_0,J)$.
{\color{blue}
%We will consider in more detail the case separating contour $\gamma$.
}
The graph $\Gamma_2$ and jump matrices on its edges are  constructed via amalgamation of the following three graphs:
\begin{enumerate}
    \item 
    The graph $\t\Gamma_2$, constructed as follows.
    Fix arbitrarily an orientation of the edges in the    triangulation  $ \t{\Sigma} $ of the capping\footnote{Recall that,  given a surface with boundaries, its capping (or capped surface) is the surface without boundary obtained by pasting back a disk for each connected component of the boundary.}  of  $\t{\Ccal}$,  with a single vertex $ \t{v}$ such that the boundary $\gamma$ and contour $\t \g$ fall into one of the triangles.   To each edge $\t e$ of the triangulation $\t \Sigma$ we associate a parameter $z_{\t e}\in \C^\times$ which we call (complex) ``shear-type coordinate''. We will set  $\zeta_{\t e}=\ln \t z_e$ (ditto for edges of $\h \Sigma$). Our final goal is to show that the symplectic form has constant coefficients when expressed in the $\zeta$ - coordinates: for this reason the choice of branch of the logarithm will not play any practical role.
      The number of edges of $ \t{\Sigma} $ equals $ 6 \t{g}-3$. The jump matrices on  the edges $\t e$ of  $ \t{\Sigma} $ are chosen as follows
  \be
    J(\t e)={{S}}_{\t e}= \begin{pmatrix}
    0 &  {z}_{\t e} \\
    -\frac 1{{z}_{\t e} }& 0 \\
    \end{pmatrix}, \quad \quad {z}_{\t e} \in \C^\times  \;.
\label{defS_e}
   \ee   
    Note\footnote{Relative to Section 7 of \cite{bertola2023extended}, we adopt a different convention so that the $z_e$ here correspond to the $\frac 1 {z_e}$ there.} that $  S_{ \t e}^{-1}=    - S_{ \t e} =   S_{- \t e}$. 
    
   For each triangle, $\t F_k$, of  $ \t{\Sigma} $  we add an extra internal vertex $ \t{w}_k$ connected to the three instances of $ \t{v}$ on the boundary of $\t F_k$ by three  isotopically distinct internal edges oriented toward $ \t{w}_k$.
   To each of these newly added internal edges (shown in red in Fig.\ref{Graph2} and Fig.\ref{flatpic}) we assign  the constant jump matrix: 
    \be
    \label{defAm}
    A = \begin{pmatrix}
    0 & -1 \\
    1 & -1 \\
\end{pmatrix} \;.
\ee
Observe that $A^3=\Id$, which ensures the condition \eqref{JJJ} at the newly added vertex $\t{w}_k$.

    \item
  The graph $\h\Gamma_2$ is constructed in completely analogous fashion on the $\h\Ccal$ component.
 \item The ``gluing''  5-vertex graph part   $ \Gamma_{pl}$ is as in Fig.\ref{graphs}, with  the jump matrices on its edges 
    defined in terms of the variables $\t \zeta_{e_j}$ and $\h \zeta_{e_j}$ as follows:
    
 \begin{enumerate}
 \item 
   On the edge $[\t{v}, \t{q}]$ we set  $ J([\t{v}, \t{q}]) = M_{\t{\g}}$, where  $M_{\t{\g}}$ is  computed imposing  the condition \eqref{JJJ} at the vertex $\t{v}$, by evaluating the inverse of the product of the other jump matrices.

   The  valence of $\t v$ is  $n_{\t v} = 2(6\t g-3)$ since each edge has both ends at $\t v$. Starting from the first occurrence of an edge of the triangulation $\t \Sigma$  in counterclockwise order from the edge $[\t v , \t q]$ (see Fig.\ref{graphs}) we encounter the same edge twice (but not consecutively). We use  the following convention:  by ``half-edge'' at a vertex $v$ we mean the part of the edge that belongs to a small disk centered at $v$ and denote it by  $h$. For a given vertex $v$  order the half-edges in cyclic order,  $h_1,\dots, h_{n_v}$ and then denote by $z_{h_j}$ (by slight abuse of notation) the shear-type coordinate associated to the edge $e$ the half-edge $h_j$ belongs to. This way, in the cyclic order, the same shear coordinate appears twice (with opposite signs due to $S^{-1}= -S$). We adopt the same liberty for the indexing of the matrices $S_{h_j}$. Then, 
    after a simple computation we find that $M_{\t\g}$ is  lower-triangular
\be
    M_{\t{\g}} = \left( \prod_{l=1}^{12\t g -6} A {S}_{\t h_l} \right)^{-1}= \begin{pmatrix}
    -\lambda_{\t{\g}} & 0 \\
    * & - \lambda_{\t{\g}}^{-1} \end{pmatrix}\;,
    \label{Mtg}
\ee
    where the matrices $S_e$'s are defined in \eqref{defS_e}, $A$ in \eqref{defAm}  and the eigenvalue is related to the shear coordinates as follows:
    \be
    \lambda_{\t{\g}}=  \prod_{\t e\in \t \Sigma}   z^2_{\t e}\;,\qquad\ \ \ell_{\t{\gamma}} = 2\sum_{\t e\in \t \Sigma} {\zeta}_{\t e} \mod 2\pi i\;.
    \la{consintro}
    \ee
     The product (and sum) is over the $6\t g-3$ edges of the triangulation $\t \Sigma$ and the minus sign in \eqref{Mtg} comes from the fact that in the product we have an odd number ($6g-3$) of inverse matrices $S_{e}$, keeping in mind that $S_e^{-1} =-S_e$. 
\item
    On the closed edge $\t \g = [\t{q}, \t{q}]$ (see Fig.\ref{graphs} and Fig.\ref{GluingGraph})   we set $J( \t \g) = C_{\t \g}$ where 
   ${C}_{\t \g}$ is a  diagonalizing matrix  for $M_{\t{\g}}$: we normalize the eigenvector by imposing that  it has the form
    $$
     C_{\t \g}=\begin{pmatrix}
    1 & 0 \\
    * & 1 \end{pmatrix}\;.
$$
The general    $ SL(2, \C) $ matrix ${\Cp}_{\t \g} $ diagonalizing $ M_{\t{\g}}$ 
has therefore  the form
    \be
     {\Cp}_{\t \g}= {C}_{\t \g}  \begin{pmatrix}
    b_{\t \g}^{-1} & 0 \\
    0 & b_{\t \g}\end{pmatrix} \;,
    \la{CCt} 
    \ee
    where the complex number $b_{\t \g}$ will be referred to as the  {\it toric variable}.  The corresponding logarithmic coordinate is denoted by $ \beta_{\t \g }=\ln{b_{\t \g}}$.
\item
    A set of formulas  completely parallel to \eqref{Mtg}, \eqref{consintro}, \eqref{CCt} can be repeated for the $\h{\Ccal}$ side, thus defining $M_{\h{\g}} ,  \lambda_{\h{\g}},  C_{\h \g}$.
    \end{enumerate}

\item We require that the shear coordinates   satisfy the  constraint $\l_{\t\g} = \l_{\h\g}$:
  \be
    \prod_{j=1}^{6 \t{g}-3}   z^2_{\t e_j} =\prod_{i=1}^{6 \h{g}-3}   z^2_{\h e_i} \;,
    \la{zz}
 \ee
    or, in terms of their logarithms 
        \be
    \sum_{j=1}^{6 \t{g}-3}   \zeta_{\t e_j}   =  \sum_{i=1}^{6 \h{g}-3} \zeta_{\h e_i} \ \mod{\pi i} .
    \la{zetazeta}
    \ee
    \end{enumerate}

Note that the   condition $
 \lambda_{\h{\g}}= \lambda_{\t{\g}}
 $   guarantees the no-monodromy condition \eqref{JJJ} at the vertex $q$:
\be
\t{\Lambda} \bb \h{\Lambda} \bb^{-1} = \Id. 
\ee
In the sequel we  denote 
$$
\lambda_{\g} :=\lambda_{\t{\g}}=  \lambda_{\h{\g}}\;,  \hskip1.2cm  \len_\g := \ln \lambda_{\g}\;.
$$
An example of graph $\Gamma_2$ for genus 2 surface $\Ccal$ is shown in  Fig.\ref{Graph2}.

\begin{figure}[!ht]
\centering
\includegraphics[scale=1]{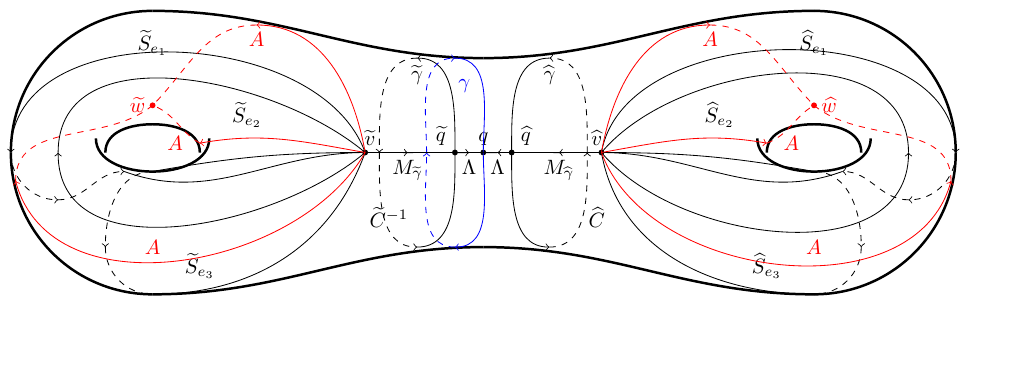}
\caption{The graph $ \Gamma_2 $ drawn on $ \Ccal$ for $ g=2$.}
\label{Graph2}
\end{figure}

Another representation of $ \Gamma_2$, with explicitly shown triangulation,  is given schematically in Fig.\ref{flatpic}. To get from Fig.\ref{Graph2} to Fig.\ref{flatpic} we cut the genus two surface $\Ccal$ along the contour $ \g$ (shown in blue) positioned  between $\t{\g}$ and $\h{\g}$. Then the edge $e$ connecting $\t q$ and $\h q$ is cut in half, leaving the half-edges inside the small circles in both genus one components  
shown in Fig.\ref{flatpic}.

\begin{figure}[h]
\centering
\includegraphics[scale=1]{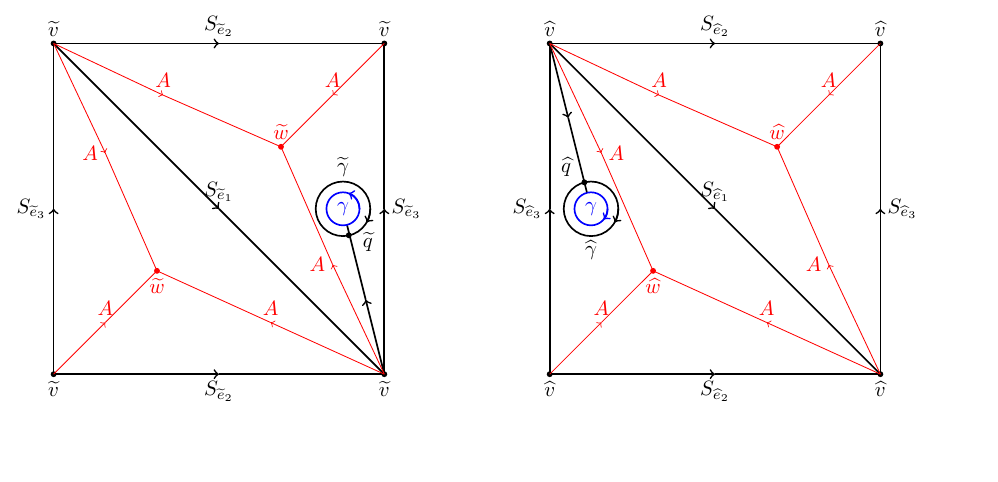}
\caption{Triangulations of genus one components $ {\Ccal}_{\t g} $ and   $ {\Ccal}_{\h g} $ for $g=2$. The segments 
    inside of the circles represent the two edges $[\t{q}, q]$ and $ [q, \h{q}]$ in Fig.\ref{Graph2}.}
    \label{flatpic}
\end{figure}

The graph $ \Gcal_2 $ can be transformed into the graph $ \Gcal_1$ (defined in Sec.\ref{Sec5.1} and shown in an example in Fig.\ref{GluingGraph})  by the following sequence of admissible moves:
\begin{enumerate}
    \item Merge  the  internal vertices $ \t{w}_k$ with $ \t{v}$ and $ \h{w}_k $ with $ \h{v}$.
    \item Zip together the edges with jump matrices $ A$ and the edges with jump matrices $ {S}_{\t e_i}$ (resp. $ {S}_{\h e_i}$) to get edges with jump matrices $ A {S}_{\t e_i} $ or $ A^{-1} {S}_{ \t e_i}$ (resp. $ A {S}_{\h e_i} $ or $ A^{-1} {S}_{\h e_i}$).
    \item Zip together  the edges of the triangulations $ \t{\Sigma} $ and $ \h{\Sigma} $ to get  the set of edges $ \{ \t{\alpha}_i, \t{\beta}_i \}_{i=1}^{\t{g}}$ and  $ \{ \h{\alpha}_i, \h{\beta}_i \}_{i=1}^{\h{g}}$. The jump matrices 
    \be
     \{ M_{\t{\alpha}_j}\;,  M_{\t{\beta}_j}\;,  M_{\h{\alpha}_j}\;,  M_{\h{\beta}_j} \}
 \la{MMM2}\ee
    on these new edges become then some products of $ {S}_{\t e_i} $ $  {S}_{\h e_i} $, $A$ and $A^{-1}$. 
\end{enumerate}

The monodromy matrices $M_{\alpha_j}$, $ M_{\beta_j}$ can be obtained from the matrices (\ref{MMM2}) using formulas (\ref{monC}) with $C_{\h\g}$ and $C_{\t\g}$ replaced by the general diagonalizing matrices $ C'_{\t\g}$ and $ C'_{\h\g}$ (\ref{CCt}).

The  matrix
\be
{C_{\t \g}'}^{-1}  \begin{bmatrix}
0&-1\\1&0
\end{bmatrix}
{C'_{\h \g}}=  C_{\t \g}^{-1}  \begin{bmatrix}
0&-b_{\t \g}\,  b_{\h \g} \\ \left(b_{\t \g} \, b_{\h \g}\right)^{-1}&0
\end{bmatrix} {C_{\h\g}} 
\ee
depends only on the product of the toric variables;   thus one  can define the new logarithmic toric variable
\be
\beta_\g=2(\beta_{\t \g}+ \beta_{\h \g})\;,  \hskip0.7cm  \beta_{\t\g} = \log b_{\t\g},\;, \hskip0.7cm  \beta_{\h\g} = \log b_{\h\g}\;.
\la{betadef}
\ee
where the factor $2$ is  introduced for future convenience.

\subsection{Graph $\Gamma_2$ in case of  non-separating contour}
\la{CtCh2}
We now have a contour $\gamma$ such that $\Ccal\setminus \gamma$ is a Riemann surface $ \Ccal_{g-1}$ of genus $ g-1$ with two boundary components. Choose  $\t v, \h v$ in the same  annular neighbourhood of $\gamma$ and on the either  sides of it  and two more loops $\t \g, \h \g$ homotopically equivalent to $\g$ . We perform a triangulation, $\Sigma$, of the capping of  $ \Ccal_{g-1}$ and with two vertices at $\t v,\h v$. For each face of this triangulation we select a point and add three additional internal edges from it, exactly as in step $1$ of section \ref{CtCh1}, resulting in a graph $\Gamma^{(g-1)}_2$. The matrices on its edges  are defined in the same way as in Section \ref{CtCh1}.   The final graph $\Gamma_2$ is then obtained, similarly, by gluing $\Gamma^{(g-1)}_2$ with the same plumbing graph $\Gamma_{pl}$ as in Section \ref{CtCh1} and shown in Fig. \ref{graphs}. 
In order to perform the gluing we need to impose a log-linear constraint on the  complex shear coordinates;
\be
\label{loglin}
\prod_{e\perp \t v \atop e\in \Sigma }(-1)^{o_e(\t v)}  z_e^{\mu_e(\t v)}  =- \lambda_\g  = \prod_{e\perp \h v \atop e\in \Sigma}(-1)^{o_e(\h v)}  z_e^{\mu_e(\h v)}\;,
\ee
where $o_e(v) = 1$ if the oriented edge  of $\Sigma$ has one end terminating at $v$ and $o_e(v)=0$ otherwise\footnote{ To be completely clear, an edge may have both ends at a vertex $v$, in which case  $o_e(v)=1$; if only one end is at $v$ and that is oriented away then $o_e(v)=0$, while if it is oriented towards $v$ then $o_e(v)=1$.}.
Here $\mu_e({v)} \in \{1,2\}$ denotes the multiplicity of the edge $e$ at the vertex $v$, namely, it is  the number of endpoints of $e$ coinciding with the vertex $v $.  

\begin{figure}[h]
\centering
\includegraphics[scale=1]{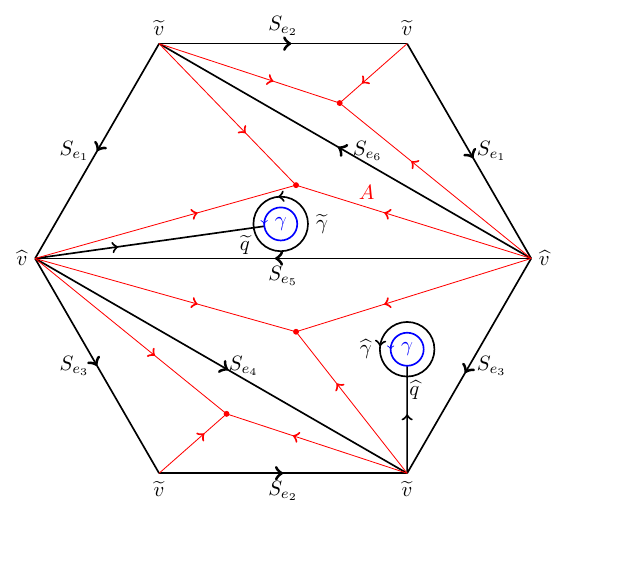}
\caption{Triangulations of a genus one component $ {\Ccal}_{g-1} $  with two boundary components to be glued. The segments 
    inside of the circles represent the two edges $[\t{q}, q]$ and $ [q, \h{q}]$ in Fig.\ref{graphs}, while the blue circles are going to be identified. In this case the relation \eqref{loglin} reads 
    $z_{e_1}z_{e_3}z_{e_4}z_{e_5}^2z_{e_6} = z_{e_1}z_{e_2}^2z_{e_3}z_{e_4}z_{e_6}$.}
    \label{flatpic2}
\end{figure}
 
Similarly to the separating case, we can deform the graph $\Gamma_2$ to $\Gamma_1$ by  a series of admissible transformations, and, therefore,
the Goldman form on ${\mathcal V}_g^{\gamma}$ can be represented as $\Omega(\Gamma_2)$ and further expressed in terms of shear coordinates and toric variables as we show in the next section.

\section{Goldman symplectic form  and  $\Omega(\Gcal_2)$}

\subsection{Case of separating contour}

We now show that the Goldman symplectic form has constant coefficients when expressed in terms of the logarithmic shear coordinates and additional coordinates  $\ell_\g$ and  $\beta _\g$. 
For a given vertex $v$  of valence $n_v$ we denote $e^{(v)}_j$ , $j=1,\dots, n_v$ the edges incident at $v$ enumerated in the cyclic order, with the understanding that a returning edge appears twice, so that for example $e_j^{(v)} = e_k^{(v)} = e$ if the same edge $e\in E(\Gamma)$  returns to $v$. Recall the triangulations $\t \Sigma ,\h \Sigma$ defined in Section \ref{CtCh1}.

\begin{theo} \label{maintheo}
Let $\zeta_{e}= \ln z_e$ for any edge $e$ and set  $\ell_\g = \ln \lambda_\g, \ \ \beta_\g = \ln b_\g$. Suppose that the logarithms satisfy the linear constraints 
\be \la{cons3}
 \len_\gamma =2\sum_{ \t e\in \t \Sigma}   \zeta_{ \t e}=2\sum_{\h e\in \h\Sigma }  \zeta_{\h e}\;.
\ee
Then Goldman  form  $\Omega$ on $ \V_g$ can be written as follows in terms of logarithmic variables, subject to  (\ref{cons3}):
\be
  \Omega=\Omega_{0}+ {\Omega_{\t v}}+{\Omega_{\h v} } 
   \la{WPint2}
  \ee
where 
$$
\Omega_0=\mathrm{d}\beta_\gamma \wedge \mathrm{d}\len_\gamma  \;,
$$
$$
{\Omega_{\t v}}= \sum_{\substack{i,j=1 \\ i<j} }^{12 \t{g}-6} 
 \mathrm{d}\zeta_{\t h_i} \wedge \mathrm{d} \zeta_{\t h_j}\;,\hskip0.7cm
{\Omega_{\h v}}=\sum_{\substack{k,l=1 \\ k<l} }^{12 \h g-6}  \mathrm{d}\zeta_{\h h_k} \wedge \mathrm{d}\zeta_{\h h_l}\;,
$$
and $\t h_j$ are the half edges incident at $\t v$ as described above \eqref{Mtg} and enumerated counterclockwise starting from the first edge of the triangulation $\t \Sigma$ after the edge $[\t v, \t q]$ (see Fig. \ref{Graph2}). Ditto for $\h h_k$'s. 

\vskip0.5cm

\end{theo}
\begin{proof}
The symplectic form $ \Omega({\Gcal_2}) $ can be  computed using the general formula (\ref{OmegaG}) by adding up the contributions of all  four vertices $ \t{v}$, $ \h{v}$, $ \t{q}$ and $ \h{q}$ of the graph $ \Gcal_2$.
  Following the general formula \eqref{OmegaG},  we claim that the contribution  of the vertex $\t{q}$  to  $ \Omega({\Gcal_2}) $   is equal to $ \Omega_{\t{q}}=2\mathrm{d}\beta_{\t \g} \wedge \mathrm{d}\len_\gamma  $, the contribution of the vertex  $\h{q}$  is equal to $ \Omega_{\h{q}} =2\mathrm{d}\beta_{\h \g} \wedge \mathrm{d}\len_\gamma$ and the contribution of the vertex $ q$ vanishes.

To prove this claim we observe that the four edges incident to $ \t{q} $ in counterclockwise order have the following jump matrices: $ J(e_1) =J_1= M_{\t{\g}}$, $ J(e_2)=J_2 = {C'_{\t\g}}$, $ J(e_3)=J_3 = \Lambda^{-1} $ and $ J(e_4)=J_4= {C'_{\t\g}}^{-1}$. 
According to equation (\ref{OmegaG})  the contribution of  $ \t q $ to $\Omega ({\Gcal_2}) $ reads:
\begin{align*}
    \Omega_{\t{q}}&=\frac{1}{2}\sum_{l=1}^{3} \mathrm{tr} \left( J_{[1...l]}^{-1} \d J_{[1...l]} \wedge J_{l}^{-1} \d J_{l} \right)\\
    &= \frac{1}{2}\mathrm{tr}\left( (J_1 J_2)^{-1} \d(J_1 J_2) \wedge  J_2^{-1} \d J_2 \right) + \frac{1}{2} \mathrm{tr}\left( (J_1 J_2 J_3)^{-1} \d (J_1 J_2 J_3) \wedge J_3^{-1} \d J_3 \right) \\
    &= \frac{1}{2} \mathrm{tr} \left( ( M_{\t{\g}}  {C'_{\t\g}})^{-1} \d ( M_{\t{\g}}  {C'_{\t\g}}) \wedge {C'_{\t\g}}^{-1}  \d {C'_{\t\g}} \right) +\frac{1}{2}\mathrm{tr} \left( {C'_{\t\g}}^{-1} \d  {C'_{\t\g}} \wedge \Lambda  \d (\Lambda^{-1}) \right) \\
    &= \frac{1}{2} \mathrm{tr}\left( M_{\t{\g}}^{-1} \d M_{\t{\g}} \wedge {C'_{\t\g}}^{-1} \d  {C'_{\t\g}} \right) + \frac{1}{2} \mathrm{tr}\left( \Lambda^{-1} \d \Lambda \wedge  {C'_{\t\g}}^{-1} \d {C'_{\t\g}} \right).
\end{align*} 

Using the identity  $M_{\t{\g}}^{-1} \d M_{\t{\g}} =  {C'_{\t\g}} \Lambda^{-1} {C'_{\t\g}}^{-1} \left(\d  {C'_{\t\g}}+  \d \Lambda + \d ( {C'_{\t\g}}^{-1}) \right) $ we see that the first term is equal to the second one, so the contribution $ \Omega_{\t{q}} $ simplifies to:
$$ 
\Omega_{\t{q}} =  \mathrm{tr}\left( \Lambda^{-1} \d \Lambda \wedge  {C'_{\t\g}}^{-1} \d   {C'_{\t\g}}\right)\;.$$ 
Using the definition (\ref{CCt}) of  $  {C'_{\t\g}}$ we get:
$$ 
\Omega_{\t{q}}=\mathrm{tr}\left( \Lambda^{-1} \d \Lambda \wedge  {C'_{\t\g}}^{-1}  \d {C'_{\t\g}} \right) = -2 \frac{\mathrm{d}\lambda_\gamma}{\lambda_\gamma} \wedge \frac{\mathrm{d} b_{\t \g}}{b_{\t \g}}=2 \mathrm{d} \beta_{\t \g} \wedge \mathrm{d}\len_\gamma \;.
$$
Similarly,
$$
 \Omega_{\h{q}}=-2\frac{\mathrm{d}\lambda_\gamma}{\lambda_\gamma} \wedge \frac{\mathrm{d}b_{\h \g}}{b_{\h \g}}= 2 \mathrm{d}\beta_{\h \g} \wedge \mathrm{d}\len_\gamma \; .
 $$
The four edges incident to $ q$ have jump matrices $ \bb^{-1}$, $ \Lambda$, $ \bb$ and $ \Lambda$. 
Since the matrix $ \bb$ is constant  the contribution of $ q$ to $ \Omega$ is zero as can be seen from (\ref{OmegaG}).

To compute  the contributions of the vertex $ \t{v}$ we can use the result of \cite{bertola2023extended} (or simply perform the computation) which, adapted to the current notation,  states that 
   \be
   \label{omegatv} 
   \Omega_{\t{v}}= \sum_{\substack{i,j=1 \\ i <j}}^{2(6\t{g}-3)} \mathrm{d} {\zeta}_{ \t h_i}  \wedge \mathrm{d} {\zeta}_{\t h_j} \;.
   \ee
 It is important to  clarify where the cyclic order of the triangulation edges at $\t v$ starts because, {by itself}, the expression \eqref{omegatv} is not invariant under cyclic permutations of the edges. 
   The edge $e_1$ in \eqref{omegatv} is the first edge of the triangulation $\t\Sigma$ (see Section \ref{CtCh}) that appears in counter-clockwise order at $\t v$ starting from the edge $[\t v, \t q]$, see Fig. \ref{Graph2}. In \cite{bertola2023extended}, Sec.6, it was explained how the {\it total} expression for $\Omega$ is independent of the ``ciliation'' (terminology of loc. cit.), namely, (in the present context) in which triangle of the triangulation $\t\Sigma$  we place the gluing zone. However, this invariance requires that the toric variable $b_{\t \g}$ undergoes a log-linear change, see discussion in \cite{bertola2023extended}  (where the case of the  general $SL(N,\C)$ character variety was considered; in that paper the   log-toric variables are denoted by $\rho$ and the eigenvalues by $\mu$). 
    Similar expression and considerations  hold for $\Omega_{\h v}$.
   \end{proof}

\begin{remark}\rm
 Let   $ {\Hcal}_{g,1}$ be the extended $SL(2,\C)$ character variety   of Riemann surfaces of genus $g$ with one boundary homotopic to $\g$ as described in \eqref{ext}. It is a symplectic manifold, as discussed  in  Section \ref{SecExtSymp}.  
 Then  Th.\ref{maintheo} can be reformulated in the language of   Hamiltonian reduction as follows.  Consider the symplectic manifold $  \Hcal_{\t g,1}\times \Hcal_{\h g,1}$ with the product symplectic structure. Let $\len_{\t\gamma}, \len_{\h\gamma}$ be the eigenvalues of $M_{\t\gamma}, M_{\h\gamma}$, respectively, on the two factors. Recall that the Darboux conjugate coordinates are the  log-toric variables $\beta_{\t\g}, \beta_{\h\g}$. Now we can interpret the space $\V_g^{(\gamma)}$ 
 equipped with Goldman symplectic form as the Hamiltonian reduction of $\Hcal_{\t g,1}\times \Hcal_{\h g,1}$ under the action of the Hamiltonian $\mu(\len_{\t\g}, \len_{\h\g}):=\len_{\t\g}-\len_{\h\g}$, which generates the $\C$ action 
   \be
   \beta_{\t\g}\mapsto    \beta_{\t\g}+t, \ \ \    \beta_{\h\g}\mapsto    \beta_{\h\g}-t, \ \ \ t\in \C\;,
   \ee namely
   \be
   \V_g^{(\gamma)} = \mu^{-1}(\{0\}) \sslash \C \;.
   \ee
The Goldman symplectic form $ \Omega$ can be written as in Theorem \ref{maintheo} on the reduced space $  \mu^{-1}(0) \sslash \mathbb{C} $. \hfill $\triangle$. 
\end{remark}

\subsection{Log-canonical coordinates in case of  a non-separating contour}
The computation is essentially identical, with the only difference being that the number of half-edges incident at $\t v, \h v$ depends on the choice of triangulation and not just on the genus of the surface with a total number of edges constrained by the Euler characteristic to be $6g-6$, since we have two vertices of a triangulation of a  surface of genus $g-1$.

Again, the Goldman symplectic form has constant coefficients when expressed in terms of the logarithmic shear coordinates and additional coordinates  $\ell_\g$ and  $\beta _\g$, where $\beta_\gamma$ is defined in the same formula (\ref{betadef}) as in the separating case.
%For a given vertex $v$  of valence $n_v$ we denote $e^{(v)}_j$ , $j=1,\dots, n_v$ the edges incident at $v$ enumerated in the cyclic order, with the understanding that a returning edge appears twice, so that for example $e_j^{(v)} = e_k^{(v)} = e$ if the same edge $e\in E(\Gamma)$  returns to $v$. 

Recall the triangulations $\Sigma$ defined in Section \ref{CtCh2}.
An analog of Thm.\ref{maintheo} in the non-separating case can be stated  as follows, with the proof proceeding in completely parallel fashion.

\begin{theo} \label{maintheo1}
Let $\zeta_{e}= \ln z_e$ for any edge $e$ and set  $\ell_\g = \ln \lambda_\g, \ \ \beta_\g = \ln b_\g$. Suppose that the logarithms satisfy the linear constraints 
\be \la{cons31}
 \len_\gamma =\sum_{  e\perp \tilde{v},\;e\in\Sigma} \mu_e(\t v) \zeta_{  e}=\sum_{   e\perp \hat{v},\;e\in\Sigma  }  \mu_e(\h v) \zeta_{e}\;\hskip0.7cm {\rm mod }\;\; \pi i \;.
\ee
Then Goldman  form  $\Omega$ on $ \V_g$ can be written as follows in terms of logarithmic variables, subject to  (\ref{cons31}):
\be
  \Omega=\Omega_{0}+ {\Omega_{\t v}}+{\Omega_{\h v} } 
   \la{WPint21}
  \ee
where 
$$
\Omega_0=\mathrm{d}\beta_\gamma \wedge \mathrm{d}\len_\gamma  \;,
$$
$$
{\Omega_{\t v}}= \sum_{\substack{i,j=1 \\ i<j} }^{n_{\tilde{v}}} 
 \mathrm{d}\zeta_{\t h_i} \wedge \mathrm{d} \zeta_{\t h_j}\;,\hskip0.7cm
{\Omega_{\h v}}=\sum_{\substack{k,l=1 \\ k<l} }^{ n_{\hat{v}}}  \mathrm{d}\zeta_{\h h_k} \wedge \mathrm{d}\zeta_{\h h_l}\;,
$$
$n_{\tilde{v}} $   and $n_{\hat{v}}$ are valencies of $\tilde{v}$ and $\hat{v}$, respectively;
 $\t h_j$ are the half edges incident at $\t v$ as described above \eqref{Mtg} and enumerated counterclockwise starting from the first edge of the triangulation $\t \Sigma$ after the edge $[\t v, \t q]$ (see Fig.\ref{Graph2}). Ditto for $\h h_k$'s. 

\vskip0.5cm

\end{theo}

\section{Generalization to multiple cutting contours}
A system of log-canonical coordinates generalizing the previous construction can be naturally associated to any system of $m$ closed nonintersecting geodesics $\gamma_1,\dots \gamma_m$ for any $m=1,\dots, 3g-3$.
These contours split $\Ccal$ into one or more Riemann surfaces with boundaries which we denote by $\Ccal^{(1)},\dots\Ccal^{(n)}$.
Denote the number of boundary components of $\Ccal^{(i)}$ by $k^{(i)}$ and the genus by $g^{(i)}$; we have $\sum_{i=1}^n k^{(i)}= 2m$. The boundary components of each $\Ccal^{(i)}$ are denoted $ \gamma_j^{(i)}$, $ j=1,..., k^{(i)}$, see Figure \ref{MSC}.

\begin{figure}[ht!]
    \centering
    \includegraphics{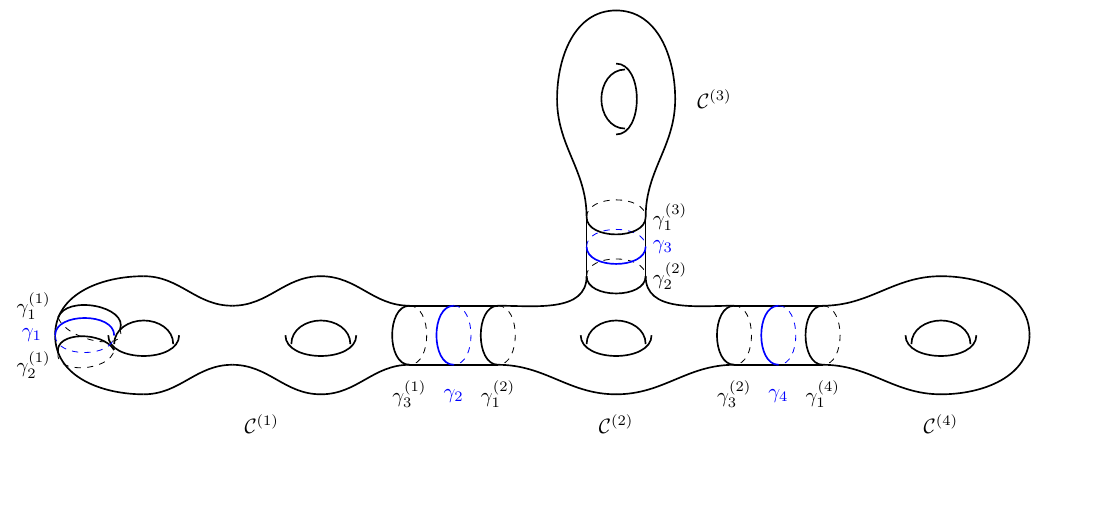}
    \caption{An example of multiple contours for $ m=4$ and $ g=5$.}
    \label{MSC}
\end{figure}
On each $\Ccal^{(i)}$ we choose points $v_a^{(i)}, \ a = 1,\dots, k^{(i)}$ near the corresponding boundary components. Let  $\Sigma^{(i)}$  be a triangulation of the surface-without-boundary obtained by capping the $k^{(i)}$ boundaries of $\Ccal^{(i)}$ with as many  disks, and with vertices at  points $v^{(i)}_a$'s. The triangulation $\Sigma^{(i)}$ has $6g^{(i)}-6+3 k^{(i)}$ edges.

We associate complex log-shear coordinates $\zeta_{e}$ to each    edge $e$ of $\Sigma^{(i)}$ (and $z_e = {\rm e}^{\zeta_e}$).
Adding internal vertices and edges for each triangle in $\Ccal^{(i)}$ we get a graph $ \Gamma^{(i)}$ on each $ \Ccal^{(i)}$.  The graph in $ \Ccal$ obtained by amalgamating all graphs $ \Gamma^{(i)}$, is denoted by $ \Gamma$.

The eigenvalues of $M_{\gamma_a^{(i)}}$ in terms of the shear coordinates are $ (-1)^{\sharp{\text {(incoming edges)}}}\big(\lambda_a^{(i)}\big)^{\pm 1} $ where 
\be
 \lambda_a^{(i)}:= \prod_{e \perp v^{(i)}_a \atop e \in E(\Sigma^{(i)})} z_e^{\mu_e({v_a^{(i)})}} , \ \ \ i = 1,\dots, n, \ \ a = 1,\dots, k^{(i)}, 
\ee 
and $\mu_e({v_a^{(i)})} \in \{1,2\}$ denotes the multiplicity of the edge $e$ at the vertex $v_a^{(i)} $, namely, it is  the number of endpoints of $e$ coinciding with the vertex $v_a^{(i)} $. 
Then the complex oriented lengths $\len_{\g_a^{(i)}}$ of $\gamma_a^{(i)}$  are related to $\zeta_e$'s: 
\bea \la{lengai}
\len_{\g_a^{(i)}}  := \sum_{e\in E(\Sigma^{(i)}),\; e\perp v_k^{(i)}}\mu_e({v_a^{(i)})}\, \zeta_e   \mod 2\pi i  \;, \ \ \ \ \ 
 i=1,\dots m , \ \ \ a = 1, \dots, k^{(i)}. 
\eea
The constraints encode the combinatorics of the gluing of the different components; if $\gamma^{(i)}_a$ and $\gamma^{(j)}_b$ belong to the same free homotopy class $[\gamma_k]$ ($k=1,\dots, m$)  in the total surface then we have to impose the constraint 
\be
 \label{constMC} 
 \len_{\g_a^{(i)}}  = \len_{\g_b^{(j)}}  = \len_{\g_k}.
\ee

The system of constraints (\ref{constMC}) is linear in terms of $\zeta_e$. 
The system of log-canonical coordinates contains ``complex lengths''  $\len_{\g_k}$ for $k=1,\dots,m$ and the shear-type variables $\zeta_e$ 
(there are linear constraints (\ref{constMC}) relating   $\len_{\g_k}$ to $\zeta_e$'s).

Define the space $\V_g^{{\boldsymbol \gamma}}$, where ${\boldsymbol \gamma}=(\gamma_1,\dots,\gamma_m)$ such that the monodromy matrices $M_{\gamma_j}$ 
corresponding to contours $\gamma_j$ are all diagonalizable. Then the following theorem holds:

\begin{theo} \label{MC}
The Goldman form $ \Omega $ on $\V_g^{{\boldsymbol \gamma}}$   can be represented as
\be
\Omega=\sum_{i=1}^n \Omega^{(i)} + \sum_{j=1}^m  \d\beta_{\g_j} \wedge \d\len_{\g_j} \;,
\la{O12}\ee
where
\be
\Omega^{(i)}=\sum_{v\in V(\Sigma^{(i)})}\sum_{\substack{e, \t{e} \perp v \\ e<\t{e}} } \d\zeta_e \wedge \d \zeta_{\t{e}} \;.
\la{Oi}\ee
 The form is restricted to the (linear) submanifold cut by the collection of all constraints \eqref{constMC}. 
\end{theo}

\begin{proof}
    The proof is by direct computation of $ \Omega$ as $ \Omega(\Gamma)$, which is a straightforward generalization of the one given in the proof of Theorem \ref{maintheo}.\end{proof}

\section{Complete trinion decomposition}

Let us discuss in detail the case of complete trinion decomposition with an eye to future applications to the parametrization of Teichm\"uller spaces. This decomposition is obtained by splitting $ \Ccal $ along a system, $ S$, of $m=3g-3$ closed non-intersecting simple loops  $ \g_1, ..., \g_{3g-3}$ to obtain $ 2g-2 $ trinions  denoted by $ \Pcal^{(j)}$, each of which is conformally equivalent to a thrice-holed sphere.

Consider a single trinion $ \Pcal$: a triangulation is shown in Figure \ref{triangle}, left pane,  with edges ordered as follows (counterclockwise from the stem $[v_j, q_j]$):
 \begin{equation} \label{edgeordering}
     \begin{split}
        & e_2, e_3 \perp v_1, \quad e_3<e_2 \;, \\
        &  e_1, e_3 \perp v_2, \quad e_1<e_3 \;,\\
        & e_1, e_2 \perp v_3, \quad e_2<e_1 \;.
     \end{split}
 \end{equation}

\begin{figure}[!ht]
    \includegraphics[width=0.53\textwidth]{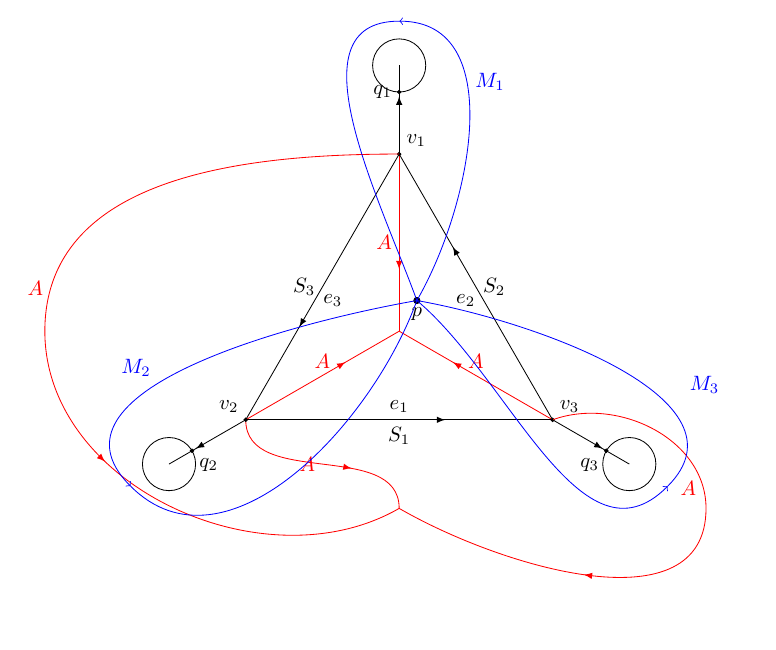}\includegraphics[width=0.4\textwidth]{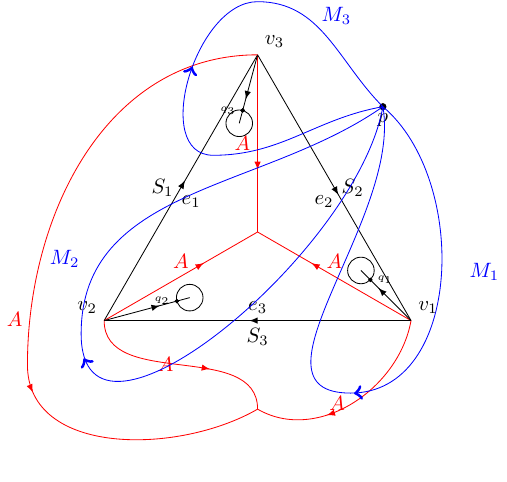}
    \caption{Triangulation and monodromies associated to a single trinion $ \Pcal$ but inducing opposite cyclic order of the boundaries.}
    \label{triangle}
\end{figure}

Taking the base point $ p$ inside the black triangle  (Fig.\ref{triangle}), the local $ SL(2, \mathbb{C})$ monodromy matrices around $ v_j$ are denoted by $ M_j$, with complex eigenvalues $ -\lambda_j$ and $ -1/\lambda_j$, $ j=1,2,3$.  Notice that in the previous computations we have denoted the eigenvalues of $M_j$ by $(\lambda_j, 1/\lambda_j)$.
The change of sign in the case of a trinion is done for convenience; this allows to preserve reality of the variables $z_j$ in the $SL(2,\R)$ case which  we are going to treat in the 
subsequent work.
The shear-type coordinates $z_{e_1}$, $ z_{e_2}$ and $ z_{e_3}$ can be expressed in terms of $ \lambda_1$, $ \lambda_2$ and $ \lambda_3$ as:
\begin{equation}\label{trincons}
    z_{e_1}=
   {\sqrt{\frac {\l_3\l_2}{\l_1}}}
    \;, \hskip 0.7cm 
    z_{e_2} = 
    {\sqrt{\frac{\l_1\l_3}{\l_2}} }
    \;, \hskip 0.7cm 
    z_{e_3} = 
    {\sqrt{\frac{\l_1\l_2}{\l_3}}}
    \;.
\end{equation}
Then a representation $ \rho: \pi_1(\Pcal, p) \rightarrow SL(2,\mathbb{C}) $ is explicitly given by the following three monodromies: 

\bea
M_{1} = 
\left[\begin{array}{cc}
-\lambda_{1}  & 0 
\\
 \frac{\lambda_{1} \lambda_{2} +\lambda_{3}}{\lambda_{1} \lambda_{3}} & -\frac{1}{\lambda_{1}} 
\end{array}\right]
,\qquad
M_{2} = 
\left[\begin{array}{cc}
\frac{\lambda_{3}}{\lambda_{1}} & \frac{-\lambda_{2} \lambda_{3} -\lambda_{1}}{\lambda_{1} \lambda_{2}} 
\\
 \frac{\lambda_{1} \lambda_{2} +\lambda_{3}}{\lambda_{1}} & \frac{-\lambda_{1} \lambda_{2}^{2}-\lambda_{2} \lambda_{3} -\lambda_{1}}{\lambda_{1} \lambda_{2}} 
\end{array}\right]
,\qquad
M_{3} = 
\left[\begin{array}{cc}
-\frac{1}{\lambda_{3}} & \frac{-\lambda_{2} \lambda_{3} -\lambda_{1}}{\lambda_{2}} 
\\
 0 & -\lambda_{3}  
\end{array}\right]
\ .
\label{Monotrin}
\eea

These monodromies satisfy 
 $$ M_1 M_2 M_3 = \mathbb{I}_2\;, \hspace{0.7cm} \mathrm{Tr}(M_j) = -\lambda_j - \frac {1} {\lambda_j}\;, \hspace{0.7cm} j=1,2,3\;. $$

The toric variables are defined relative to the lower triangular diagonalizing matrices of the ``local'' monodromies, see \eqref{Mtg}, namely, the jump matrices on the edges $[v_j,q_j]$ (indices taken cyclically):
\bea
M_j^{(loc)}  = \big(A S_{e_{j+2}} A S_{e_{j+1}}\big)^{-1} = \left[\begin{array}{cc}
-\lambda_{j}  & 0 
\\
 \frac{\lambda_{j} \lambda_{j+1} +\lambda_{j+2}}{\lambda_{j+2} \lambda_{j}} & -\frac{1}{\lambda_{j}} 
\end{array}\right],
\eea
so that the diagonalizing matrices can be normalized as follows:
\bea
M_j^{(loc)} = C_j \Lambda C_j^{-1} ,\ \ \ \ C_j= \left[\begin{array}{cc}
1 & 0 
\\
 \frac{-\lambda_{j} \lambda_{j+1} -\lambda_{j+3}}{\left(\lambda_{j}^{2}-1\right) \lambda_{j+3}} & 1 
\end{array}\right]\;.
\eea
Consider now Fig. \ref{triangle}, right pane: this is topologically the same graph  on the Riemann sphere but now the cyclic order of the labelling of the three boundaries is the opposite so that the holes $1,3$ on the left pane are now labelled $3,1$, respectively (and the shear coordinates are  similarly  re-labelled). Marking the shear coordinates of the right trinion with a tilde, we have $z_1 = \t z_3,\  z_2 =\t z_2, \ z_3 = \t z_1$ and similarly $\lambda_1 = \t \lambda_3, \ \lambda_2=\t \lambda_2, \ \lambda_3 = \t \lambda_1$. The local monodromies and diagonalizing matrices for the left trinion are related to the ones on the right trinion by 
\be
\t M_3^{(loc)} = \big(A S_{e_3}\big)^{-1}M_1^{(loc)}A S_{e_3},\ \ 
\ee
so that the diagonalizing matrix is $ C_1 \big(A S_{e_3}\big)^{-1}$, 
and similar for $\t M_2^{(loc)},\t M_1^{(loc)}$. Computing directly  gives
\be
 C_1 \big(A S_{e_3}\big)^{-1} = \left[\begin{array}{cc}
\sqrt{\frac {\t \lambda_{3} \t \lambda_{2}}{\t \lambda_{1}}}& 0 
\\
 -\frac{\sqrt{\frac{\t \lambda_{1} \t \lambda_{3}}{\t \lambda'_{2}}}\, \left(\t \lambda_{1} \t \lambda_{3} +\t \lambda_{2} \right)}{\left({\t \lambda_{3}}^{2}-1\right) \t \lambda_{1}} &\sqrt{\frac{\t \lambda_{1}}{\t \lambda_{3} \t \lambda_{2}}} 
\end{array}\right]
=
\underbrace{\left[\begin{array}{cc}
1 & 0 
\\
 -\frac{ \left(\lambda_{3} \lambda_{1} +\lambda_{2} \right)}{\left(\lambda_{1}^{2}-1\right) \lambda_{2}} &1
\end{array}\right]}_{\t C_3}
 \left(
 \frac{\lambda_{3}} {\lambda_{1} \lambda_{2}}\right)^{\frac {-\sigma_3}2}\;.
\ee
We thus see that the diagonalizing matrix $C_1 \big(A S_{e_3}\big)^{-1},$ relative to the normalized matrix $\t C_3$, is shifted by the right toric action. Therefore the ``new'' toric variables $\t b_j$ are in the following relation with the ``old'' ones, $b_j$ (repeating a computation similar to the one above  for the other two matrices as well):
\be
\label{toricshift}
\t b_3 = \sqrt{ \frac{\lambda_{3}} {\lambda_{1} \lambda_{2}}} b_1\ ,\ \ \ 
\t b_2 =  \sqrt{ \frac{\lambda_{1}} {\lambda_{2} \lambda_{3}}} b_2\ ,\ \ \ 
\t b_1 =  \sqrt{ \frac{\lambda_{2}} {\lambda_{3} \lambda_{1}}} b_3\ .
\ee
For the log-toric variables the shift reads
\be
\label{logtoricshift}
\t\beta_3 = \beta_1 + \frac 1 2\left(\len_3 - \len _1 - \len _2\right), 
\ \ \ 
\t\beta_2 = \beta_2 + \frac 1 2\left(\len_1 - \len _2 - \len _3\right), 
\ \ \ 
\t\beta_1 = \beta_3 + \frac 1 2\left(\len_2 - \len _1 - \len _3\right).
\ee
For later reference we point out that (indices taken cyclically):
\be
\label{savinggrace}
\sum_{j=1}^3 \d \t \beta_j\wedge \d \t\len _j =   \sum_{a=1}^3 \d \len_a \wedge \d \len_{a+1}  + \sum_{j=1}^3 \d  \beta_j\wedge \d \len _j.
\ee

The exact expression of the Goldman form $ \Omega$  depends on the combinatorics of the  trinion decomposition which is encoded  in the trivalent ``trinion graph'' corresponding to a complete trinion decomposition.
\begin{dref}\rm
\label{deftringraph}
    The {\it ribbon trinion graph} $\mathfrak T$  of a complete trinion decomposition is the tri-valent graph whose vertices correspond to  individual trinions and the three edges meeting at the same 
    vertex are ordered cyclically.
     The edges of $\mathfrak  T$ connect two vertices if the corresponding trinions share a boundary component. To avoid confusion with the edges of other graphs, the edges of the trinion graph will be denoted $\frak e \in E( \mathfrak T)$. Accordingly, the loops $\gamma_j$ can be indexed by  the edges of the trinion graph: $\gamma_j\sim\gamma_{\frak e}$ for $ \ \frak e\in E(\frak T)$. 
\end{dref}

Let the boundary components of $ \Pcal^{(j)}$ be $ \g_1^{(j)}$, $ \g_2^{(j)}$ and $ \g_3^{(j)}$, ordered according to the cyclic ordering of the  edges of the graph $\mathfrak T$ meeting at $ \Pcal^{(j)}$, 
see Fig. \ref{triangle}.  The choice of this cyclic order is equivalent to the  choice of triangulation of the trinion as follows: the triangulation shown in Fig. \ref{pants} determines two regions  and one of the two is a triangle, while the other contains the three boundary components. Then the labelling of the components is cyclic relative to the triangle, as shown ibidem.

\begin{theo} \label{TheoTD}
    Let $ {\mathfrak T} $ be the trinion graph of a complete trinion decomposition. Then the Goldman form on $ \V_g^{{\boldsymbol \gamma}}$ can be expressed as:
  \begin{equation}\label{Omtrin}
    \begin{split}
       \Omega=   \sum_{\mathcal T^{(j)}\in V(\frak T) } (\d \len^{(j)}_1\wedge\d \len^{(j)} _2 + \d \len^{(j)} _2\wedge\d  \len^{(j)}_3+ \d \len^{(j)}_3\wedge \d \len^{(j)}_1)
        + \sum_{\frak e \in E(\frak T)} \mathrm{d} \beta_{\frak e}  \wedge \mathrm{d} \len_{\frak e} \;,
    \end{split}   
    \end{equation}
    with the  understanding that the form is restricted according to the following rule:
 For each edge $\frak e=[j,k]$ that connects the boundary $a$ of  $\mathcal T^{(j)}$ and the boundary $b$ of $\mathcal T^{(k)}$  ($a, b\in \{1,2,3\}$) we impose the constraint $\d \ell_{a}^{(j)} =  \d \ell_{\mathfrak e}$ and $\d \ell_b^{(k)} = \d \ell_{\mathfrak e}$.

Here the $3g-3$ variables $ \len_{\frak e} $  are the ``complex oriented lengths'' of the edges of the trinion graph, and $ \beta_{\frak e}$'s are the toric  variables, see Figure \ref{pants}.
\end{theo}
\begin{proof}
    Let $ \{ \lambda_a^{(j)}, b_a^{(j)} \}_{a=1,2,3} $ be the eigenvalues and toric variables associated with the $ j$-th trinion $ \Pcal^{(j)}$. Their logarithmic counterparts are $ \{ \len_a^{(j)}, \beta_a^{(j)} \}_{a=1,2,3} $.  
        An explicit computation using  (\ref{OmegaSt}) with  the monodromies (\ref{Monotrin}) yields the following contribution of an individual trinion $ \Pcal^{(j)}$ to the symplectic form $ \Omega $. This contribution is  coming from the  vertices $v_1, v_2, v_3, q_1,q_2,q_3$ shown in Fig. \ref{triangle}:
    \begin{equation} \label{omPj}
        \Omega_{\Pcal^{(j)}}=\sum_{a=1}^{3} \left( \mathrm{d} \len_{a+1}^{(j)} \wedge \mathrm{d} \len_a^{(j)} +  2 \mathrm{d} \beta_a^{(j)} \wedge \mathrm{d} \len_a^{(j)} \right)\;,
    \end{equation}
    where the summation index $ a$ is  taken modulo $3$.
    
    The form \eqref{omPj} must be summed over all trinions and then subjected to the constraints that we now describe, encoding the combinatorics of the trinion graph. 
        For each edge $\mathfrak e=[j,k]\in E(\mathfrak T)$ connecting the boundary labelled $a$ of the trinion $\mathcal T^{(j)}$ to the one labelled $b$ of trinion $\mathcal T^{(k)}$,  as explained in the statement of the theorem, the  gluing constraints require :
    \begin{equation}\label{consPj}
        \len_a^{(j)}=\len_b^{(k)}= \len_{e} \mod 2 \pi i\;, \hskip0.7cm 
     \end{equation}
    
    The corresponding toric variables are defined according to (\ref{betadef}):
    
\bea
    \beta_{\mathfrak e}= 2 \left( \beta_a^{(j)}+\beta_b^{(k)} \right) \;.
\la{betadef2}
 \eea

    For a loop edge $ \frak e= [j,j]$ incident to a vertex $\mathcal T^{(j)}$ at the boundaries $a, a+1$ ($a=1,2,3$),  the constraints become:
    $$ 
    \len_a^{(j)}= \len_{a+1}^{(j)}= \len_{\mathfrak e} \;, \hskip0.7cm  
    \beta_{\mathfrak e }=2 \left( \beta_a^{(j)}+\beta_{a+1}^{(j)} \right) \;.
    $$
 
     Summing over all vertices $ v$ of the trinion graph and using the ordering (\ref{edgeordering}), we obtain
    \begin{equation*}
        \Omega= \sum_{v \in V(\mathfrak T)} \sum_{\substack{\mathfrak e, \t{\mathfrak e} \perp v \\ \t{\mathfrak e} < {\mathfrak e}}} \mathrm{d} \len_{\mathfrak e} \wedge  \mathrm{d} \len_{\t{\mathfrak e}} + \sum_{\mathfrak e \in E(\mathfrak  T)} \mathrm{d} \beta_{\mathfrak e}  \wedge \mathrm{d} \len_{\mathfrak e}\;.
    \end{equation*}
   \end{proof}

\section{Conclusion and future directions}

We mention three main directions towards which to extend the present results.

The first one regards application to the moduli space of Riemann surfaces $M_g$. This requires the restriction of the present construction to the
Fuchsian component of the real slice of the character variety $V_g$; such a restriction requires the characterization of the Fuchsian component given by Goldman \cite{gold1988}.
For complete trinion decomposition characterized by a chosen trinion graph $ \Gamma_{trin}$, the relationship to Fenchel-Nielsen coordinates on $ \Mcal_g$ will be established in \cite{LogCharacterPaper2}. These new systems of log-canonical coordinates as well as generating functions may find several potential applications. 

{The second direction is to understand the relationship between coordinate systems associated with different choices of geometrical data of our construction, i.e., the set of dissection contours on the original surface and the triangulations of each component of the dissection. 
Under a change of triangulation of a given component the corresponding shear coordinates transform according to a suitable sequence of cluster algebra mutations. In this case the generating function of the associated symplectomorphism is a combination of dilogarithms \cite{bertola2023extended}. It is less clear at the moment how the change of dissection affects the coordinate system: since these changes induce symplectomorphisms as well, the natural question that arises is to describe the corresponding generating function. 

}

The third is the extension to the $SL(N,C)$ character varieties of compact Riemann surfaces of genus $g$ and to their real slices (in particular, higher Teichmueller spaces) for any $N\geq 3$.
Such extension would also be based on log-canonical  coordinate systems on extended character varieties of punctured Riemann surfaces \cite{bertola2023extended} which in turn uses the
higher rank Fock-Goncharov coordinates \cite{fock2006moduli}.
The jump matrices on edges $S_e$ then become anti-diagonal; each such matrix contains $N-1$ coordinates while the matrices $A$ depend on  $(n-1)(n-2)/2$ additional coordinates (one coordinate for $SL(3)$ case, three for $SL(4)$, etc. ) \cite{bertola2023extended}. The logarithms of the eigenvalues of the monodromy matrices around the separating contours and the logarithms of the ratios of the corresponding toric variables form an additional set of $N-1$ conjugated pairs of coordinates for the Goldman symplectic form.
The technical details for an arbitrary $SL(N)$ case are very similar to \cite{bertola2023extended}: the full analysis is deferred to a future publication.
\appendix
\section{Examples}\renewcommand{\theequation}{\Alph{section}.\arabic{equation}}

\la{examplesapp}
In this appendix we compute explicitly, by way of illustration of our general construction,   the forms of Theorem \ref{maintheo} and Theorem \ref{MC} on $ \Mcal_{2}$ (with one, two and three contours) and on $ \Mcal_{3}$  (with one, two, three and four contours).
\subsection{$ \V_{2}$}\subsubsection{One separating contour}
For $\V_{2}$ the idea is to select a separating contour $ \g $ which cuts a double-torus into two one-holed tori $ \t{\Ccal}$ and $ \h{\Ccal}$, endowed with graphs $ \t{\Gamma}$ and $ \h{\Gamma}$ as shown in Figure \ref{fig:201c} (the internal edges and the contour $ \g $ are not represented), whose associated shear type coordinates are $\{ \t \zeta_{e_1}, \t \zeta_{e_2}, \t \zeta_{e_3} \} $ and $ \{ \h \zeta_{e_1}, \h \zeta_{e_2}, \h \zeta_{e_3} \} $, respectively. \\
\begin{figure}[ht]
\centering
\includegraphics{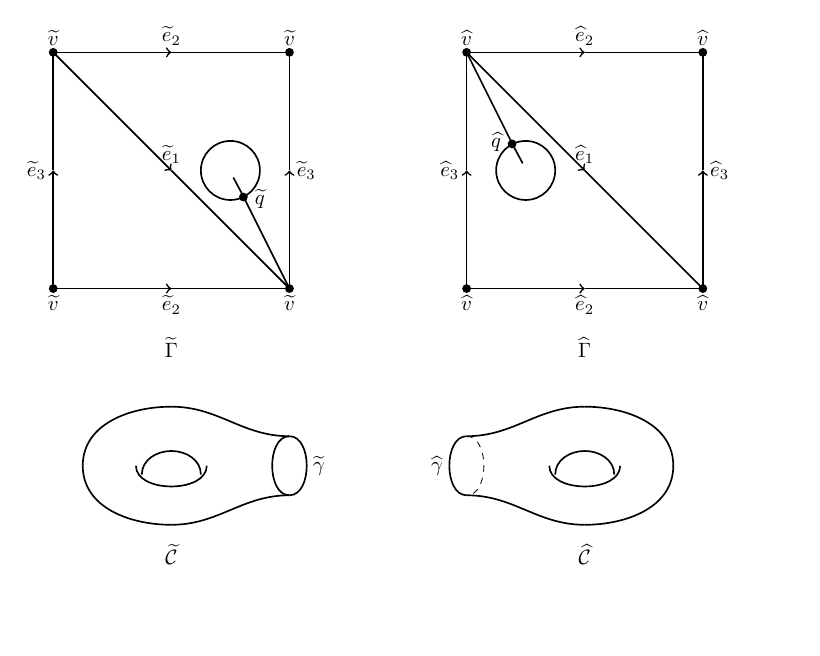}
\caption{Single contour splitting for $ \V_{2}$ and associated graphs.}
  \label{fig:201c}
\end{figure}

\begin{prop}
    In log-canonical coordinates $\{  \zeta_{\t e_2},  \zeta_{\t e_3},  \zeta_{\h e_2},  \zeta_{ \h e_3}, \len_{\g}, \beta_{\g} \}$ the Goldman form on $\V_{2} $ can be written:
    \begin{equation}\label{M201C}
        \omega_{G}=  2\mathrm{d}  \zeta_{\t e_2} \wedge \mathrm{d}  \zeta_{\t e_3} +  2 \mathrm{d}  \zeta_{\h e_2}  \wedge \mathrm{d}  \zeta_{\h e_3}+ \mathrm{d} \len_{\g} \wedge \left( \mathrm{d}  \zeta_{\t e_2} +  \mathrm{d}  \zeta_{\t e_3} \right)  + \mathrm{d} \len_{\g} \wedge \left( \mathrm{d}  \zeta_{\h e_2}   + \mathrm{d}  \zeta_{\h e_3} \right) +  \mathrm{d} \beta_{\g} \wedge \mathrm{d}\len_{\g}\;.
    \end{equation}
\end{prop}
\begin{proof}
    Using equation (\ref{WPint21}) for the vertex $ \t v$ of $ \t{\Gamma}$ in Figure \ref{fig:201c}, together with the twist-length type contribution $  2 \mathrm{d}  \beta_{\t \g} \wedge \mathrm{d} \len_{\t \g}$ coming from $ \t q$, without any gluing constraint, we get:
    \begin{equation*}
        \Omega(\t \Gamma)=2 \mathrm{d}  \zeta_{\t e_1} \wedge \mathrm{d}  \zeta_{\t e_2}+2 \mathrm{d}  \zeta_{\t e_1} \wedge \mathrm{d}  \zeta_{\t e_3}+2 \mathrm{d}  \zeta_{\t e_2} \wedge \mathrm{d}  \zeta_{\t e_3}+ 2
          \mathrm{d} \beta_{\t \g} \wedge \mathrm{d}  \len_{\t \g}\;.
    \end{equation*}
    Similarly for the vertex $ \h v$ of $ \h{\Gamma}$, together with the twist-length type contribution $  2 \mathrm{d} \beta_{\h \g} \wedge \mathrm{d} \len_{\h \g}$ coming from $ \h q$, we obtain:
    \begin{equation*}
        \Omega(\h \Gamma)= 2 \mathrm{d}  \zeta_{\h e_1} \wedge \mathrm{d}  \zeta_{\h e_2}+2  \mathrm{d}  \zeta_{\h e_1} \wedge \mathrm{d}  \zeta_{\h e_3}+2 \mathrm{d}  \zeta_{\h e_2} \wedge \mathrm{d}  \zeta_{\h e_3} + 2 
        \mathrm{d} \beta_{\h \g} \wedge \mathrm{d} \len_{\h \g}\;.
    \end{equation*}

Eliminating $\zeta_{\t e_1}, \zeta_{\h e_1}$ from the constraint equations $\ell_\gamma =2 \sum_{j=1}^3\zeta_{\t e_j} = 2\sum_{j=1}^3\zeta_{\h e_j} $, we obtain the final expression \eqref{M201C}.
 \end{proof}
\subsubsection{One separating and one non-separating contour}

Here one of the  contours, $ \gamma_2$, is non-separating. The resulting pieces $\t \Ccal$ and $\h \Ccal$ by cutting along $ \g_1$ and $ \g_2$ are a one-holed torus and a three-holed sphere (i.e. a trinion) respectively, see Figure \ref{fig:202c}. We can associate to $ \t \Ccal$, endowed with $\t \Gamma$, shear-type coordinates $ \{\t \zeta_{e_1}, \t \zeta_{e_2}, \t \zeta_{e_3} \}$ whereas $\h \Ccal$, with $\h  \Gamma$, comes with shear-type coordinates $ \{\h \zeta_{e_1}, \h \zeta_{e_2}, \h \zeta_{e_3} \}$, but all of them can be replaced by length-twist coordinates $ \{ \len_{\g_1}, \len_{\g_2}, \beta_{\g_1}, \beta_{\g_2} \}$ due to the constraints. 
\begin{prop}\label{M202C}
    In log-canonical coordinates $\{ \zeta_{\t e_2},  \zeta_{\t e_3}, \len_{\g_1}, \len_{\g_2}, \beta_{\g_1}, \beta_{\g_2} \}$ the Goldman  form on $\V_{2} $ can be written :
    $$ \omega_{G} =  2 \mathrm{d} \zeta_{\t e_2} \wedge \mathrm{d}  \zeta_{\t e_3} + \mathrm{d} \len_{\g_1}  \wedge \left(\mathrm{d}  \zeta_{\t e_2} + \mathrm{d} \zeta_{\t e_3}  \right)  +  \mathrm{d} \beta_{\g_1} \wedge \mathrm{d} \len_{\g_1}  + \mathrm{d} \beta_{\g_2} \wedge  \mathrm{d} \len_{\g_2} \;.$$
\end{prop}
\begin{figure}[ht!]
  \centering
\includegraphics{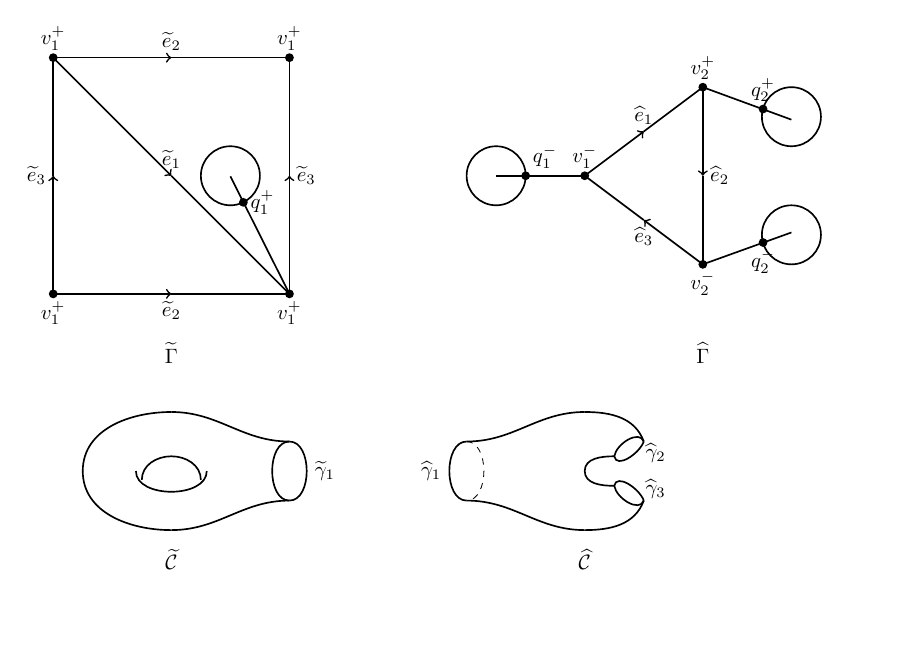}
\caption{Two contour splitting for $ \V_{2}$ and associated graphs.}
\label{fig:202c}
\end{figure}
\begin{proof}
    As in the previous example we still have (the contribution from $ q_1^+$ is now given by $ 2 \mathrm{d} \beta_{\t \g_1} \wedge  \mathrm{d} \len_{\t \g_1} $):
    \begin{equation} \label{Om22C1}
         \Omega(\t \Gamma)=2 \mathrm{d}  \zeta_{\t e_1} \wedge \mathrm{d}  \zeta_{\t e_2}+2 \mathrm{d}  \zeta_{\t e_1} \wedge \mathrm{d}  \zeta_{\t e_3}+2 \mathrm{d}  \zeta_{\t e_2} \wedge \mathrm{d}  \zeta_{\t e_3}+
          2 \mathrm{d} \beta_{\t \g_1} \wedge  \mathrm{d} \len_{\t \g_1} \; ,
    \end{equation}
 the vertices $ v_1^-$,  $ v_2^-$, $ v_2^+$ of $ \h \Gamma $ (together with twist-length type contributions coming from $ q_1^-$,  $ q_2^-$, $ q_2^+$) contribute:
    \begin{equation} \label{Om22C2}
        \Omega(\h \Gamma)= \mathrm{d}  \zeta_{\h e_3} \wedge \mathrm{d}  \zeta_{\h e_1} + \mathrm{d}  \zeta_{\h e_1}  \wedge \mathrm{d}  \zeta_{\h e_2} + \mathrm{d}  \zeta_{\h e_2} \wedge \mathrm{d}  \zeta_{\h e_3}+ 2\mathrm{d} \beta_{\h \g_1} \wedge \mathrm{d} \len_{\h \g_1}  + 2\mathrm{d} \beta_{\h \g_2} \wedge  \mathrm{d} \len_{\h \g_2} + 2\mathrm{d} \beta_{\h \g_3} \wedge \mathrm{d} \len_{\h \g_3}\;.
    \end{equation}
    The constraints for $ \gamma_1$ reads (mod $\pi i$):
    \begin{align*}
        2( \zeta_{\t e_1}+ \zeta_{\t e_2}+ \zeta_{\t e_3}) &= \len_{\t \g_1} \; ,\qquad
         \zeta_{\h e_1}+  \zeta_{\h e_3} = \len_{\h \g_1} \;,\qquad
        \len_{\t \g_1} =  \len_{\h \g_1} = \len_{ \g_1} \;,
    \end{align*}
    and the constraints for the non-separating contour $ \gamma_2$ are given by:
    \begin{align*}
         \zeta_{\h e_1} + \zeta_{\h e_2} &= \len_{\h \g_2} \;,\qquad 
         \zeta_{\h e_3} +  \zeta_{\h e_2} = \len_{\h \g_3}\;,\qquad
        \len_{\h \g_2} =  \len_{\h \g_3} = \len_{\g_2} \;.
    \end{align*}
    Hence,
    \begin{align*}
         \zeta_{\t e_1} &=\frac{1}{2} \len_{\g_1}- \zeta_{\t e_2}-  \zeta_{\t e_3} \;, \hspace{0.7cm}  \zeta_{\h e_2} = \len_{\g_2} -\frac{1}{2} \len_{\g_1} \;, \\
         \zeta_{\h e_1} & = \frac{1}{2} \len_{\g_1} \;, \hspace{0.7cm} \zeta_{\h e_3}  =\frac{1}{2}\len_{\g_1} \;.
    \end{align*}
    Inserting the constraints in (\ref{Om22C1})  and (\ref{Om22C2}) we obtain:
    \begin{align*}
        \Omega(\t \Gamma)&= 2 \mathrm{d}  \zeta_{\t e_2}  \wedge \mathrm{d}  \zeta_{\t e_3} + \mathrm{d} \len_{\g_1} \wedge \left(\mathrm{d}  \zeta_{\t e_2} + \mathrm{d}  \zeta_{\t e_3} \right)+ 2\mathrm{d} \beta_{\t \g_1} \wedge \mathrm{d} \len_{\g_1} \;,  \\
        \Omega(\h \Gamma) &= 2\mathrm{d} \beta_{\h \g_1} \wedge \mathrm{d} \len_{\g_1} + 2\mathrm{d} \beta_{\h \g_2} \wedge \mathrm{d} \len_{\g_2}  +   2\mathrm{d} \beta_{\h \g_3}  \wedge \mathrm{d} \len_{\g_2} \;.
    \end{align*}
    Summing up $\Omega(\t \Gamma) $ and $ \Omega(\h \Gamma)$ while setting $ \beta_{\g_1}= 2( \beta_{\t \g_1}+\beta_{\h \g_1})$ and $ \beta_{\g_2} = 2(\beta_{\h \g_2}+\beta_{\h \g_3})$ we get the expression of Proposition \ref{M202C}.
\end{proof}

\subsubsection{Three separating  contours: trinion decomposition}
In this situation, we get all complete decompositions into trinions. The three possible ribbon trinion graphs $ \t {\mathfrak T}, \t{ \mathfrak T'} $ and $  \h {\mathfrak T}$ are shown in Fig.\ref{V2TD}. The log-canonical coordinates associated to $ \t {\mathfrak T}$ (resp. $ \t {\mathfrak T'},\  \h {\mathfrak T}$)  are denoted $ \{ \len_{\t {\mathfrak e_i}}, \beta_{\t {\mathfrak e_i}} \}_{i=1}^3$ (resp. $ \{ \len_{\h{\mathfrak e_i}}, \beta_{\h{\mathfrak e_i}} \}_{i=1}^3$ $ \{ \len_{\t{\mathfrak e'_i}}, \beta_{\t{\mathfrak e'_i}} \}_{i=1}^3$). The two ribbon graphs  $ \t {\mathfrak T}, \t{ \mathfrak T'} $  are related by the change of cyclic ordering of the edges at one of the two trinions.

\begin{prop}
The expression for the Goldman  form $ \omega_{G} $ on $ \V_{2}$ has the following expressions  in the  coordinates $ \{ \len_{\t {\mathfrak e_i}}, \beta_{\t  {\mathfrak e_i}} \}_{i=1}^3$, $ \{ \len_{\t {\mathfrak e'_i}}, \beta_{\t  {\mathfrak e'_i}} \}_{i=1}^3$, $ \{ \len_{\h  {\mathfrak e_i}}, \beta_{\h  {\mathfrak e_i}} \}_{i=1}^3$ corresponding to the different trinion graphs  $ \t {\mathfrak T}, \t{ \mathfrak T'} , \h {\mathfrak T}$, respectively: 
\begin{itemize}
    \item $ \Omega(\t{\mathfrak T})= \mathrm{d} \beta_{\t  {\mathfrak e_1}} \wedge \mathrm{d} \len_{\t  {\mathfrak e_1}}   + \mathrm{d} \beta_{\t  {\mathfrak e_2}} \wedge \mathrm{d} \len_{\t  {\mathfrak e_2}}   + \mathrm{d} \beta_{\t  {\mathfrak e_3}} \wedge \mathrm{d} \len_{\t  {\mathfrak e_3}}  $\;, \\
    \item $ \Omega(\t{\mathfrak T'})=2\Big[ \d \len_{\t  {\mathfrak e'_3}} \wedge \d \len_{\t  {\mathfrak e'_2}} +  \d \len_{\t  {\mathfrak e'_1}} \wedge \d \len_{\t  {\mathfrak e'_3}} +  \d \len_{\t  {\mathfrak e'_2}} \wedge \d \len_{\t  {\mathfrak e'_1}}\Big]+ 
     \mathrm{d} \beta_{\t  {\mathfrak e'_1}} \wedge \mathrm{d} \len_{\t  {\mathfrak e'_1}}   + \mathrm{d} \beta_{\t  {\mathfrak e'_2}} \wedge \mathrm{d} \len_{\t  {\mathfrak e'_2}}   + \mathrm{d} \beta_{\t  {\mathfrak e'_3}} \wedge \mathrm{d} \len_{\t  {\mathfrak e'_3}}  $\;, \\
    \item $ \Omega(\h {\mathfrak T})= \mathrm{d} \beta_{\h  {\mathfrak e_1}} \wedge \mathrm{d} \len_{\h  {\mathfrak e_1}}   + \mathrm{d} \beta_{\h  {\mathfrak e_2}} \wedge \mathrm{d} \len_{\h  {\mathfrak e_2}}   + \mathrm{d} \beta_{\h  {\mathfrak e_3}} \wedge \mathrm{d} \len_{\h  {\mathfrak e_3}}  $\;. 
\end{itemize}
\end{prop}
\begin{figure}[ht!]
    \centering
    \includegraphics{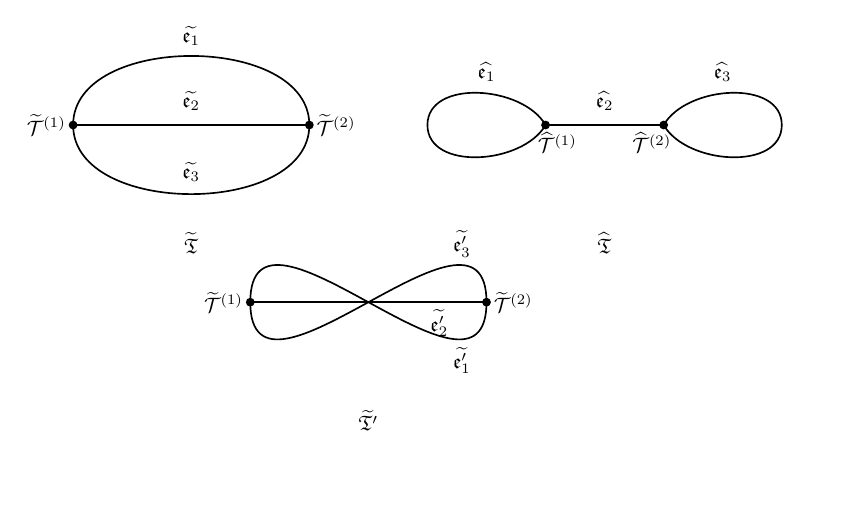}
    \caption{The ribbon  trinion graphs $  \t{\mathfrak T}, \t{\mathfrak T'}$ and $ \h{\mathfrak T}$.}
    \label{V2TD}
\end{figure}

\begin{proof}
    The proof is by direct application of Theorem \ref{TheoTD} to $ \t{\mathfrak T},  \t{\mathfrak T'}$ and $\h{\mathfrak T}$.

    Using formula (\ref{Omtrin}) for $\t{\mathfrak T}$ one can get:
    \begin{align*}
        &\Omega(\t{\mathfrak T}) =  \Big[ \d \len_{\t  {\mathfrak e_3}} \wedge \d \len_{\t  {\mathfrak e_2}} +  \d \len_{\t  {\mathfrak e_1}} \wedge \d \len_{\t  {\mathfrak e_3}} +  \d \len_{\t  {\mathfrak e_2}} \wedge \d \len_{\t  {\mathfrak e_1}}\Big]+  \Big[ \d \len_{\t  {\mathfrak e_1}} \wedge \d \len_{\t  {\mathfrak e_2}} + \\
        &+\d \len_{\t  {\mathfrak e_2}} \wedge \d \len_{\t  {\mathfrak e_3}} +  \d \len_{\t  {\mathfrak e_3}} \wedge \d \len_{\t  {\mathfrak e_1}} \Big] + \d \beta_{\t  {\mathfrak e_1}} \wedge  \d \len_{\t  {\mathfrak e_1}} +  \d \beta_{\t  {\mathfrak e_2}} \wedge  \d \len_{\t  {\mathfrak e_2}} +  \d \beta_{\t  {\mathfrak e_3}} \wedge  \d \len_{\t  {\mathfrak e_3}} \\
        &= \d \beta_{\t  {\mathfrak e_1}} \wedge  \d \len_{\t  {\mathfrak e_1}} +  \d \beta_{\t  {\mathfrak e_2}} \wedge  \d \len_{\t  {\mathfrak e_2}} +  \d \beta_{\t  {\mathfrak e_3}} \wedge  \d \len_{\t  {\mathfrak e_3}} \;.
    \end{align*}
    For $\t {\mathfrak T'}$ we obtain:
    \begin{align*}
        &\Omega(\t{\mathfrak T'}) =  \Big[ \d \len_{\t  {\mathfrak e'_3}} \wedge \d \len_{\t  {\mathfrak e'_2}} +  \d \len_{\t  {\mathfrak e'_1}} \wedge \d \len_{\t  {\mathfrak e'_3}} +  \d \len_{\t  {\mathfrak e'_2}} \wedge \d \len_{\t  {\mathfrak e'_1}}\Big]
        + 
\Big[ \d \len_{\t  {\mathfrak e'_2}} \wedge \d \len_{\t  {\mathfrak e'_1}} + \\
        &+\d \len_{\t  {\mathfrak e'_3}} \wedge \d \len_{\t  {\mathfrak e'_2}} +  \d \len_{\t  {\mathfrak e'_1}} \wedge \d \len_{\t  {\mathfrak e'_3}} \Big] + \d \beta_{\t  {\mathfrak e'_1}} \wedge  \d \len_{\t  {\mathfrak e'_1}} +  \d \beta_{\t  {\mathfrak e'_2}} \wedge  \d \len_{\t  {\mathfrak e'_2}} +  \d \beta_{\t  {\mathfrak e'_3}} \wedge  \d \len_{\t  {\mathfrak e'_3}} \\
        &= 2\Big[ \d \len_{\t  {\mathfrak e'_2}} \wedge \d \len_{\t  {\mathfrak e'_1}} +\d \len_{\t  {\mathfrak e'_3}} \wedge \d \len_{\t  {\mathfrak e'_2}} +  \d \len_{\t  {\mathfrak e'_1}} \wedge \d \len_{\t  {\mathfrak e'_3}} \Big]
        +  \d \beta_{\t  {\mathfrak e'_1}} \wedge  \d \len_{\t  {\mathfrak e'_1}} +  \d \beta_{\t  {\mathfrak e'_2}} \wedge  \d \len_{\t  {\mathfrak e'_2}} +  \d \beta_{\t  {\mathfrak e'_3}} \wedge  \d \len_{\t  {\mathfrak e'_3}} \;.
    \end{align*}
    The two forms $\Omega (\t{\mathfrak T})$ and $\Omega (\t{\mathfrak T'})$ coincide given that the corresponding  variables are related by $\len_{\t{\mathfrak e_j}} = \len_{\t{\mathfrak e_j'}} $ and 
\bea
\label{logtortrin}
\beta_{\t {e_1'}} = \beta_{\t {e_1}}  + \left(\len_{\t {e_3}}  -\len_{\t {e_1}} -\len_{\t {e_2} }\right), 
\cr
\beta_{\t {e_2'}} = \beta_{\t {e_2}}  +\left(\len_{\t {e_1}}  -\len_{\t {e_3}} -\len_{\t {e_2} }\right), 
\cr 
\beta_{\t {e_3'}} = \beta_{\t {e_3}}  +  \left(\len_{\t {e_2}}  -\len_{\t {e_1}} -\len_{\t {e_3} }\right).
\eea

    The relation  \eqref{logtortrin} (together with $\len_{\t{\mathfrak e_j}} =\len_{\t{\mathfrak e'_j}}$)  is a consequence of the relation \eqref{logtoricshift} together with the choice of labelling and  the definition of the amalgamated toric variable \eqref{betadef}, or equivalently \eqref{betadef2}. The formulas \eqref{logtortrin}  imply that  the forms $\Omega(\t{\mathfrak T'})$ and $\Omega(\t{\mathfrak T})$ coincide, using \eqref{savinggrace} (the factor of $2$ is coming from the definition of the amalgamated toric variable in \eqref{betadef2}).

    Similarly, for $ \h  {\mathfrak T}$:
     \begin{align*}
        & \Omega(\h{\mathfrak T}) =  \Big[ 
        \d \len_{\h  {\mathfrak e_2}} \wedge \d \len_{\h  {\mathfrak e_1}} +  
        \d \len_{\h  {\mathfrak e_1} }\wedge \d \len_{\h  {\mathfrak e_1}} +  \d \len_{\h  {\mathfrak e_1}} \wedge \d \len_{\h  {\mathfrak e_2}} \Big]+ \Big[ \d \len_{\h {\mathfrak  e_2}} \wedge \d \len_{\h  {\mathfrak e_3}} + \\
        & \d \len_{\h  {\mathfrak e_3}} \wedge \d \len_{\h  {\mathfrak e_3}} +  \d \len_{\h  {\mathfrak e_3}} \wedge \d \len_{\h  {\mathfrak e_2}}\Big] + \d \beta_{\h  {\mathfrak e_1}} \wedge  \d \len_{\h  {\mathfrak e_1}} +  \d \beta_{\h  {\mathfrak e_2}} \wedge  \d \len_{\h  {\mathfrak e_2}} +  \d \beta_{\h  {\mathfrak e_3}} \wedge  \d \len_{\h  {\mathfrak e_3}} \\
        &= \d \beta_{\h  {\mathfrak e_1}} \wedge  \d \len_{\h  {\mathfrak e_1}} +  \d \beta_{\h  {\mathfrak e_2}} \wedge  \d \len_{\h  {\mathfrak e_2}} +  \d \beta_{\h  {\mathfrak e_3}} \wedge  \d \len_{\h  {\mathfrak e_3}} \;.
    \end{align*}
\end{proof}


\begin{thebibliography}{10}

\bibitem{alekseev1995symplectic} A .Yu. Alekseev, A. Z. Malkin, {\it Symplectic structure of the moduli space of flat connection on a Riemann surface}, Comm.  Math. Phys., {\bf 169} No.1 99-119 (1995)

\bibitem{bertola2023extended} M. Bertola, D. Korotkin, {\it Extended Goldman symplectic structure in Fock--Goncharov coordinates}, J.Diff.Geom., {\bf 124} No.3 397-442 (2023)

\bibitem{LogCharacterPaper2} M. Bertola, D. Korotkin, J. Pillet, {\it Log-canonical coordinates on Teichm\"uller spaces}, in preparation. 

\bibitem{boalch2018wild} P. Boalch, {\it Quasi-Hamiltonian geometry of meromorphic connections}, Duke Math. J., {\bf 139} 
No. 2, 369-405 (2007)

\bibitem{bonahon2018goldman} F. Bonahon,  I. Kim, {\it The Goldman and Fock-Goncharov coordinates for convex projective structures on surfaces}, Geometriae Dedicata, {\bf 192} N0.1 43-55 (2018)

\bibitem{chekhov2023symplectic} L. Chekhov, M. Shapiro, {\it Symplectic groupoid and cluster algebras}, arXiv:2304.05580 (2023)


\bibitem{fock2006moduli} V. Fock, A. Goncharov, {\it Moduli spaces of local systems and higher Teichm{\"u}ller theory}, Publications Math{\'e}matiques de l'IH{\'E}S {\bf 103}, 1-211 (2006)


\bibitem{jeffrey1994extended} L. Jeffrey, {Extended moduli spaces of flat connections on Riemann surfaces}, Math.Annalen, {\bf 298} 667-692 (1994)


\bibitem{goldman1984symplectic} W. M. Goldman,  {\it The symplectic nature of fundamental groups of surfaces}, Adv. in Math. {\bf 54} No.2 200-225 (1984)

\bibitem{gold1988} W. M. Goldman,  (1988). {\it Topological components of spaces of representations} Inventiones mathematicae, 93(3), 557-607.

\bibitem {thurston} W. Thurston, {\it  The geometry and topology of three-manifolds}, Princeton University Notes, http://www.msri.org/publications/books/gt3m.

\bibitem{wolpert1982} S. Wolpert \textit{The Fenchel-Nielsen deformation.} Annals of Mathematics, 115(3), (1982) 501-528.


\bibitem{wolpert1985weil} S. Wolpert. \textit{On the Weil-Petersson geometry of the moduli space of curves.} Amer. J. of Math. \textbf{107} No.4 (1985): 967-997.

\bibitem{Platis} J. R. Parker, I. D. Platis, {\it Complex hyperbolic Fenchel-Nielsen coordinates},
Topology {\bf 47} (2008), no. 2, 101-135.


\end{thebibliography}
\end{document}